\documentclass[9pt,reqno]{amsart}
\usepackage{a4wide}
\usepackage[utf8x]{inputenc}
\usepackage{ucs}
\usepackage{amsmath,amssymb,paralist,amsthm,array,enumitem}

\newtheorem{theorem}{Theorem}

\newtheorem{definition}{Definition}
\newtheorem{lemma}[definition]{Lemma}
\newtheorem{corollary}[definition]{Corollary}

\theoremstyle{definition}

\newtheorem{remark}[definition]{Remark}

\newtheorem*{theorem*}{Theorem}
\newtheorem*{definition*}{Definition}
\newtheorem*{proposition*}{Proposition}
\newtheorem*{lemma*}{Lemma}

\newcommand{\hol}{\mathrm{hol}}
\newcommand{\T}{\mathbb{T}}
\newcommand{\R}{\mathbb{R}}
\newcommand{\N}{\mathbb{N}}
\newcommand{\Z}{\mathbb{Z}}

\newcommand{\I}{{[0,1)}}
\newcommand{\A}{\mathsf{A}}
\newcommand{\delexp}{{\delta _{\mathrm{e}}}}

\newcommand{\D}{\mathcal{D}}

\newcommand{\nullset}{\mathcal{N}}
\newcommand{\J}{\mathcal{J}}

\newcommand{\dist}{{\operatorname{dist}}}

\newcommand{\graph}{\mathcal{G}}

\newcommand{\ttheta}{{\boldsymbol\vartheta}}

\title[Fractal dimensions and local Hölder exponent spectra]{Fractal dimensions of graph of Weierstrass-type~function and local~Hölder~exponent~spectra}

\author{Atsuya Otani}

\thanks{This work was inspired by workshops in the framework of the DFG Scientific Network ``Skew Product Dynamics and Multifractal Analysis'' organised by Tobias Oertel-J\"ager. The research was funded by the German RAesearch Foundation (DFG Ke 514/8-1).}
\date{\today}

\begin{document}
\begin{abstract}
We study several fractal properties of the Weierstrass-type function
\[
	W (x)=\sum _{n=0} ^\infty\lambda (x)\lambda (\tau x)\cdots\lambda (\tau ^{n-1}x)\, g(\tau ^n x),
\]
where $\tau :[0,1)\to[0,1)$ is a cookie cutter map with possibly fractal repeller, and $\lambda$ and $g$ are functions with proper regularity.
In the first part, we determine the box dimension of the graph of $W$ and Hausdorff dimension of its randomised version.
In the second part, the Housdorff spectrum of the local Hölder exponent is characterised in terms of thermodynamic formalisms.
Furthermore, in the randomised case, a novel formula for the lifted Hausdorff spectrum on the graph is provided.
\end{abstract}
\maketitle


\section{Introduction}

Let $\T:=\I$ be the interval, topologically identified with the 1-dimensional unit torus $\R/\Z$, and $(I _i) _{i\in\Sigma _\ell}$ be disjoint subintervals of $\T$, where $\Sigma _\ell =\{0,\ldots ,\ell -1\}$ for some $\ell\in\N$.
A map $\tau :\T\to\T $ is called a cookie cutter if $\tau(x)=0$ for $x\in\T\setminus\bigcup _{i\in\Sigma _\ell }I _i$, and each restriction $\tau _{|I _i ^\circ}:I _i^\circ\to (0,1)$ is a $ C ^{1+\alpha} $-diffemomorphism. In this note, we study several fractal geometrical structures of the deterministic as well as the randomised Weierstrass-type function $W_\ttheta :\T\to\T$,
\begin{equation}\label{eq:W_def}
	W_\ttheta(x)=\sum _{n=0} ^\infty\lambda (x)\lambda (\tau x)\cdots\lambda (\tau ^{n-1}x)\, g(\tau ^n x +\vartheta _n),
\end{equation}
where $\lambda :\T\to(0 ,1)$ and $g:\T\to\R$ are continuous maps which are piecewise $C^{1+\alpha}$, i.e. their restrictions to each $I _i^\circ$ are of class $C^{1+\alpha}$. In addition, the variable $\ttheta\in\T^{\N_0}$ is only used as a random sequence.
In the deterministic case, we study $W:=W_{\mathbf{0}}$.
Moreover, we always assume the following partial hyperbolic condition
\begin{equation}\label{eq:partial_hyperbolic}
		\inf _{i\in\Sigma _\ell}\inf _{x\in (I _i)^\circ}\left|\tau '(x)\right|\lambda (x)>1.
\end{equation}
Note that the classical Weierstrass function is given by $\tau (x)=\ell x\bmod 1$ and $g(x)=\cos(2\pi x)$ for some $\ell\in\N$ and $\lambda\in(0,1)$ with $\ell\lambda >1$.

Since all summands of the absolutely convergent sum in \eqref{eq:W_def} are continuous on $\T$, so is $W_\ttheta$.
In fact, $W_\ttheta$ is $\alpha$-Hölder continuous but generally, due to \eqref{eq:partial_hyperbolic}, no better regularity can be expected.
This point will be deeply discussed in terms of a multifractal analysis of local Hölder exponent.

Before proceeding, we introduce a few notation. The intervals $ ( I _i ) _{i\in\Sigma _\ell} $ are called the monotonicity intervals of $\tau $.
Furthermore, the set of singular points and the repeller of $\tau $ are respectively defined as
\begin{equation*}
	\nullset =\bigcup _{n\in\mathbb{N}}\tau ^{-n}\{ 0 , 1\}\quad\mbox{and}\quad	\J =\bigcap _{n\in\N }\overline{\tau ^{-n}(0,1)}.
\end{equation*}
Clearly, $\tau (\nullset )=\tau^{-1}(\nullset )=\nullset$ and $\tau(\J )\subseteq\J$.
When we consider the restricted dynamics $(\J,\tau_{|\J})$, we often write $\tau'$ and $\lambda$ etc. instead of $(\tau_{|\J})'$ and $\lambda _{|\J}$ etc.
This restriction is essential from the point of view of the fractality, since $W_\ttheta$ is 'smooth' outside of $\J$ as shown below.

\begin{lemma}	\label{lem:smooth_part}
If $\J\neq\T$, then the restriction of $W _\ttheta$ to $\T\setminus (\J\cup\nullset)$ is of class $C^{1+\alpha}$.
\end{lemma}\label{lem}

\begin{proof}
Let $x\in\T\setminus (\J\cup\nullset)$.
Then, we have $x\in U:=\T\setminus\overline{\tau ^{-n}(0,1)}$ for some $n\in\N$.
Observe that $U$ is a neighbourhood of $x$ such that
\[
	W _\ttheta (u)=\sum _{j=0} ^{n-1}\lambda (u)\lambda (\tau u)\cdots\lambda (\tau ^{j-1} u)g(\tau ^j u+\theta _j)+\lambda (u)\cdots\lambda (\tau ^n u)\sum _{k=0} ^\infty (\lambda (0))^k g(\theta _{k+n}) 
\]
for all $u\in U$.
Moreover, as $x\not\in\nullset$, there is a sub-neighbourhood $x\in V\subseteq U$ such that the restrictions of $\tau ,\ldots ,\tau ^n$ to $V$ are of class $C ^{1+\alpha}$.
Thus $W _\ttheta$ is continuously differentiable in $x$.
\end{proof}

Hence we are only interested in the regularity of $W_\ttheta$ over $\J$.
The following important dichotomy will be shown in a later section.
For the classical Weierstrass function, the nowhere differentiability was proved by G. H. Hardy in \cite{Hardy16}.

\begin{lemma}	\label{lem:degenerate}
$W$ is nowhere differentiable on $\J$ or is of class $C^{1+\alpha}(\T )$.
\end{lemma}

The above lemma was essentially shown in \cite{Bedford89}, where the author referred to the latter case as the degenerate case. 
For some practical reasons, in this note, we use a slightly weaker definition: Given $\ttheta$, $W_\ttheta$ is {\it degenerate} if it is Lipschitz continuous.

Our study covers two main topics. One of them concerns the Housdorff and box dimension of the graph $ \graph W_\ttheta =\{(x,W_\ttheta (x)):x\in\T\} $.
More generally, for any function $\phi :D\to\R$, we define $\graph \phi = \{ (x,\phi(x)) : x\in D\}$.
Theorem \ref{thm:dim} gives the box dimension for $W$, while Theorem \ref{thm:graph_dim_random} does the Hausdorff dimension for the randomised function $W_\ttheta$.
In the case $\J =\T $ and $\tau ' > 0 $, the box dimension of the graph of $ W $ was proved in \cite{Bedford89}, while its Hausdorff dimension in the classical case is proved in \cite{Shen15}, \cite{Barani14} and \cite{Keller2017}.
In addition, the randomised function $ W _\ttheta $ was studied in \cite{Hunt98} and \cite{Moss12}.

The other topic is the local Hölder exponent, that is defined by
\begin{equation}	\label{eq:def_holder}
	\hol _{W_\ttheta}(x)=\sup\left\{\alpha\in(0,1):\inf _{r>0}\sup _{d(x,u)<r}\frac{|W_\ttheta (x)-W_\ttheta (u)|}{d(x,u)^\alpha}<\infty\right\},
\end{equation}
where $d(x,u):=\max\{|x-u|,1-|x-u|\}$ is the torus metric.
It turns out that these values are as a function on $\J$ highly complex unless it is constant, in the multifractal point of view.
To see this, we study the level set
\[
	E _{\ttheta,\alpha} :=\left\{ x\in \J :\hol _{W _\ttheta} (x) =\alpha\right\}
\]
in terms of the Hausdorff spectrum.
For simplicity, let $ E _\alpha := E _{\mathbf{0} ,\alpha} $ for $\mathbf{0} = ( 0 , 0,\ldots ) $.
In particular, Theorems \ref{thm:dimE} and \ref{thm:random_spectrum} provide some formulas for the spectra
\[
	\alpha\mapsto\dim _H E _\alpha\quad\mbox{and}\quad\alpha\mapsto\dim _H\left(\graph W _\ttheta\cap (E _{\ttheta ,\alpha}\times\R)\right) .
\]
Note that the first spectrum for a similar complex function was studied in \cite{JS15}.
\section{Main results}

Let $ s _1 , s _2\in\R $ be the unique zeros of the Bowen equations
\begin{equation}
	P_\J ((1-s _1)\log |\tau '|+\log\lambda )=0\quad\mbox{and}\quad	P_\J(s _2\log\lambda )=0,	\label{eq:s12}
\end{equation}
where $P_\J(\cdot)$ denotes the topological pressure on $(\J ,\tau _{|\J})$.

The next theorem generalises one of the major results in \cite{Bedford89}.

\begin{theorem}	\label{thm:dim}
Suppose that $ W $ is non-degenerate. Then, we have
\[
	\dim _B\graph W_{|\J}=s _1\quad\mbox{and}\quad\dim _H\graph W_{|\J}\leqslant\min\{s _1,s _2\}.
\]
In addition, if $\J\neq\T$, then $\dim_B\graph W_{|\T\setminus\J}=\dim_H\graph W_{|\T\setminus\J}=1$.
\end{theorem}

As an immediate corollary, we can see that the Hausdorff dimension of $W_{|\J}$ may be strictly smaller than its box dimension.

\begin{corollary}
Suppose that $ W $ is non-degenerate. If $\dim _B\graph W _{|\J}<1$, then $\dim _H\graph W_{|\J}<\dim _B\graph W_{|\J}$.
\end{corollary}

Now, a natural question is whether $\dim _H\graph W_{|\J}=\min\{s _1,s _2\}$.
We will show this identity for a randomised function $W_\ttheta$ for a specific class of $g$.
Following the notation of \cite{Moss12}, we say that a function $ g :\T\to\R $ satisfies the {\it critical point hypothesis} if $ g\in C ^\infty (\T ) $ and there is some number $ r _0\in\N $ such that the orders of critical points of the functions $ g(a +\cdot ) - c g $ are strictly less than $ r _0 $ for any $ a\in (0,1) $ and $ c\in\mathbb{R} $.
For example, any non-vanishing trigonometric polynomial satisfies this condition.

\begin{theorem}	\label{thm:graph_dim_random}
Suppose that $\ttheta\in\T ^{\N _0} $ is an independent and identically uniformly distributed random sequence on $\T$ and that $\tau _{|\J}$ is expansive.\footnote{The map $\tau _{|\J}:\J\to\J$ is called expansive if $\inf \{ \sup _{n\in\N _0} d(\tau ^n x,\tau ^n v) : x,v\in\J \mbox{ with } x\neq v  \} >0 $.}
If $ g $ satisfies the critical point hypothesis, then almost surely
\[
	\dim _B\graph W _\ttheta = s _1 \quad\mbox{and}\quad\dim _H\graph W _\ttheta =\min\{ s _1 , s _2\}.
\]
\end{theorem}

In order to describe the spectrum $\alpha\mapsto\dim _H E _\alpha$, we need to introduce further notation from the thermodynamic formalism.
For each $q\in\R$, let $A _q\in\R$ be the number that is uniquely determined by
\begin{equation}	\label{eq:Bowen}
	P _\J(-A _q\log |\tau '|+q\log\lambda )=0.
\end{equation}
According to \cite{Barreira08}, we can define the following quantities:
\begin{itemize}
\item Let $\D$ be the Legendre transform of the map $ q\mapsto A _q $, i.e. $\D (\alpha ):=\sup _{q\in\R} q\alpha + A _q $.
\item Let $\alpha (q) := - A ' _q $.
\item Let $\A := (\alpha _{\min} ,\alpha _{\max} )  $, where $\alpha _{\min}:=\inf _{q\in\R}\alpha (q)$ and $\alpha _{\max}:=\sup _{q\in\R}\alpha (q)$.
\end{itemize}
Recall that, if $\A\neq\emptyset $, the restriction of $\D $ to the interval $\A $ is a strictly concave non-negative analytic function, so $\alpha _c:=\alpha(0)\in\A $ is the unique maximum point (critical point).
Observe that
\[
	\mathcal{D} (\alpha _c ) = 1\quad\mbox{and}\quad\mathcal{D} ' (\alpha _c ) = 0 .
\]
In case $\A =\emptyset$, we define $\alpha _c :=\alpha _{\min} =\alpha _{\max} $.

The Hausdorff spectrum of the local Hölder exponent is characterised as follows.
\begin{theorem}	\label{thm:dimE}
Suppose that $ W $ is non-degenerate. In case $\A =\emptyset $, we have
\[	E _\alpha = \begin{cases}
	\J &\mbox{if }\alpha =\alpha _c\\
	\emptyset  &\mbox{otherwise}
\end{cases} .
\]

In case $\A\neq\emptyset$, we have
\[
	\dim _H E _\alpha =\mathcal{D} (\alpha ) 
\]
for all $\alpha\in\A $.
Moreover, for each $\alpha\in\A $, there is a Gibbs measure $\nu _\alpha $ such that $\nu _\alpha ( E _\alpha ) = 1 $,
\[
	\dim _H\nu _\alpha =\mathcal{D} (\alpha ) =\frac{h _\tau (\nu _\alpha )}{\int\log |\tau ' |\, d\nu _\alpha} \quad\mbox{and}\quad\alpha =\frac{-\int\log\lambda\, d\nu _\alpha}{\int\log |\tau ' |\, d\nu _\alpha}.
\]
\end{theorem}

\begin{remark}
If $\lambda = |\tau ' | ^{-\theta} $ for some constant $\theta\in (0,1) $, then $\A =\emptyset$.
Clearly, then $\alpha _c =\theta $.
In particular, Theorem \ref{thm:dimE} implies \cite[Theorem 1]{Todorov15}, i.e. that the Lebesgue measure of $ E _\theta $ is one.
\end{remark}

We turn to the graph points over $E_\alpha$ for various $\alpha$.
A natural upper bound can be immediately derived from Theorem \ref{thm:dimE}, applying the general formula presented in \cite[Theorem 1]{Jin11}.

\begin{lemma}	\label{lem:Jin}
We have
\[
	\dim _H (\graph W _\ttheta\cap (E _{\ttheta ,\alpha}\times\R))\leqslant\min\left\{\dim _H E _{\ttheta ,\alpha}  + 1 -\alpha ,\frac{\dim _H E _{\ttheta ,\alpha}}{\alpha}\right\}
\]
for all $\alpha\in\R $ and $\ttheta\in\T ^{\N _0} $.
\end{lemma}

The next result on the randomised case is an application of Lemma \ref{lem:lift}, which suggests the canonical representation for the lifted spectrum.

\begin{theorem}	\label{thm:random_spectrum}
Suppose that $\ttheta\in\T ^{\N _0} $ is an independent and identically uniformly distributed random sequence on $\T$ and that $\tau _{|\J}$ is expansive.
Furthermore, suppose $\A\neq\emptyset$ and that $ g $ satisfies the critical point hypothesis.
Then, almost surely
\[
	\dim _H(\graph W _\ttheta\cap (E _{\ttheta ,\alpha}\times\R)) =\min\left\{\D (\alpha ) + 1 -\alpha ,\frac{\D (\alpha )}{\alpha}\right\} 
\]
for all $\alpha\in\A $.
\end{theorem}

\begin{remark}
If $ W $ is non-degenerate and $\A\neq\emptyset $, it is not hard\footnote{The second equation follows immediately from basic properties of $\D (\alpha)$. The first one is, however, not a trivial corollary of the theorem because Hausdorff dimension is only $\sigma$-stable.
For verification, a slight modification of its proof seems to be necessary.} to verify
\[
	\dim _B (\graph W\times (\J\times\R))=\sup _{\alpha\in \A }\min\left\{\D (\alpha ) + 1 -\alpha ,\frac{\D (\alpha )}{\alpha}\right\} =\sup _{\alpha\in \A }\left\{\dim _H E _\alpha + 1 -\alpha \right\} .
\]
\end{remark}

\section{Preliminaries}
Here is the basic notation.
Most of the definitions are related to the basis dynamics $(\tau,\T)$, while some can be only defined for the restriction $(\tau _{|\J},\J)$.
Recall that $\T:=\I$ is endowed with the torus metric $d(x,u):=\max\{|x-u|,1-|x-u|\}$.
Given an interval $ I $, let $ | I | $ and  $\partial I $ denote its length and the set of its endpoints, respectively.
That is, $ | I | := b-a $ and $\partial I :=\{ a, b\} $, where $\overline{I} =: [a,b ] $.

\begin{itemize}
\item Let $ [ x ] _n := (\kappa ( x ) ,\kappa (\tau x ) ,\ldots ,\kappa (\tau ^{n-1} x ))$ for $ x\in\J $ and $ n\in\N $, where $\kappa ( x ) := i $ for $ x\in I _i $.
\item Let $\rho _i :[0,1]\to\overline{I_i}$ be the $i$-th inverse branch for $ i\in\Sigma _\ell $, i.e. the $C^0$-extension of $(\tau _{|I_i^\circ})^{-1}$.
\item Let $\rho _{\mathbf{i}}:=\rho _{i _n}\circ\cdots\circ\rho _{i _1}$ for $\mathbf{i}=(i_1\ldots ,i_n)\in(\Sigma _\ell)^n$ and $n\in\N$.
\item The $ n $-th monotonicity interval of $ x\in\J $ is the subinterval $ I _n (x)\subseteq\T $ defined by
\[
	I _n (x) := I _{\kappa (x)}\cap\tau ^{-1} I _{\kappa (\tau x)}\cap\cdots\cap\tau ^{-(n-1)} I _{\kappa (\tau ^{n-1} x)}.
\]
\item Given any function $\phi :\J\to\R $ we simply write
\[
	\phi _n ( x ) :=\sum _{k=0} ^{n-1}\phi (\tau ^k x )\quad\mbox{and}\quad\phi ^n ( x ) :=\prod _{k=0} ^{n-1}\phi (\tau ^k x ) 
\]
for $ x\in\J $ and $ n\in\N _0 $.
\item Given Borel measurable subset $ A\subseteq\R ^d $, let $\mathcal{P} ( A ) $ denote the set of all Borel probability measures on $ A $.
Moreover, let $\mathcal{P} _\tau (\J ) $ denote the set of all $\tau $-invariant Borel measures $\nu $ on $\J $, where the $\tau $-invariance means $\nu\circ\tau ^{-1} =\nu $.
\item $\nu\in\mathcal{P} _\tau (\J ) $ is a Gibbs measure if there are constants $ C _\phi > 0 $ and $ P _\phi $ and a Hölder continuous function $\phi :\J\to\R $ such that
\[
	C _\phi ^{-1}\leqslant\frac{\nu ( I _n (x) )}{e ^{\phi _n(x) - P _\phi}}\leqslant  C _\phi
\]
for all $ x\in\J $ and $ n\in\N $.
\item Let $P(\phi)$ denote the topological pressure on $(\J ,\tau)$ with respect to a Hölder continuous potential function $\phi :\J\to\R$, i.e.
\[
	P (\phi ) =\sup _{\nu\in\mathcal{P} _\tau (\J ) } h _\tau (\nu ) +\int\phi\, d\nu ,
\]
where $ h _\tau (\nu ) $ denotes the Kolmogorov-Sinai-entropy.
\end{itemize}
For more general definitions and details of Gibbs measures and topological pressure, see \cite{Barreira08} or \cite{Pesin97}.
The next lemma summarizes several important relations between equilibrium state and Gibbs measure, more general statements are found in \cite{Keller98}.

\begin{lemma}	\label{lem:eqisGibbs}
Let $\phi :\J\to\R $ be Hölder continuous.
The equilibrium state for the potential $\phi $ is a Gibbs measure for the same potential, and vice versa, where $P _\phi =P (\phi )$ is always satisfied.
In addition, any Gibbs measure is ergodic and atom-free.
\end{lemma}

\subsection*{Hausdorff dimension of measures}
We use the standard definitions of several dimensions in \cite{Barreira08} or \cite{Pesin97}, related to sets as well as measures, so we just recall a few things here.
Let $\mu $ be a Borel measure on a metric space $(\mathcal{E},d_{\mathcal{E}})$.
The lower pointwise dimension of $\mu $ is defined as
\[
	\underline{d} _\mu (u) :=\liminf _{r\to 0}\frac{\log\nu ( B _r (u) )}{\log r} 
\]
for each $ u\in E $, where $ B _r (u) :=\{ u'\in E : d_{\mathcal{E}} ( u, u')\leqslant r\} $ denote the closed balls.
In addition, the Hausdorff dimension of the measure $\mu $, denoted by $\dim _H\mu $, is defined as
\[
	\dim _H\mu :=\inf\{\dim _H Z : Z\subseteq E\mbox{ s.t. } \mu ^\ast ( E\setminus Z ) = 0\} ,
\]
where $\mu ^\ast $ is the outer measure extension of $\mu $.

The next lemma provides an alternative definition of $\dim _H\mu $.

\begin{lemma}[{ \cite[Theorem 2.1.5 (3)]{Barreira08}}]\label{lem:Barreira_cite}
Let $\mu $ be a Borel measure on an Euclidean space $ E $.
Then $\dim _H\mu  $ is the essential supremum of $\underline{d} _\mu $ with respect to $\mu $.
\end{lemma}

\section{Dimensions of the graph}
We give proofs of Theorems \ref{thm:dim} and \ref{thm:graph_dim_random} in this section.
In addition, a short proof of Lemma \ref{lem:degenerate} can be also found here.

The following two basic lemmas are repeatedly used throughout the proof sections.

\begin{lemma}	\label{lem:tau_distortion}
There is a constant $ D > 0 $ such that
\[
	\left|\frac{(\tau^n )' (x)}{(\tau ^n) ' (u)}\right|\in [ D ^{-1} , D ],\quad  | (\tau ^n)'(u)| \cdot | I _n (x) |\in [ D ^{-1} , D ]\quad\mbox{and}\quad \frac{|\rho _{[x] _n} ' (v) |}{| I _n (x) |}\in [ D ^{-1} , D ] 
\]
for all $ u\in I _n (x) $, $ x\in\J$, $v\in\T$ and $ n\in\N $.

In addition, there is a constant $\delta _0 > 0 $ such that
\[
	\delta _0\leqslant\frac{|I _{n+1} (x)|}{| I _n (x) |}\leqslant\delta _0^{-1}
\]
for all $ x\in\J $ and $ n\in\N $.
\end{lemma}

\begin{proof}
As $\tau'$ is piecewise $\alpha$-Hölder continuous on $\bigcup _{i\in\Sigma _\ell} I_i$, so is $\log |\tau'|$.
Let $a:=\sup_{i\in\Sigma _\ell}\sup_{u\in(I_i)^\circ}|1/\tau'(u)|$ so that, in view of the mean value theorem, $ |I _k (x)|\leqslant a ^k $ for all $x\in\J$ and $k\in\N$.
The first estimate follows as
\[
	\left|\log\left|\frac{(\tau^n )' (x)}{(\tau ^n) ' (u)}\right|\right| \leqslant \sum _{j=0} ^{n-1} \left|\log |\tau'| (\tau ^j x) - \log|\tau'| (\tau ^j u) \right|\leqslant C _0\sum _{j=0} ^{n-1} | I _{n-j} (\tau ^j x ) | ^\alpha\leqslant\frac{C _0}{1- a^\alpha} ,
\]
where $ C _0 $ is the $\alpha $-Hölder constant of $\log |\tau '|$.
The second one can be derived from the first one by the mean value theorem, and the last one is a special case thereof with $u=\rho _{[x]_n}(v)$.
Finally, we can choose $\delta _0=D^{-2}a$.
The assertion is obtained by applying the second estimate twice.
\end{proof}

\begin{lemma}	\label{lem:Bedford_weak}
The followings are true.
\begin{itemize}
\item There is a constant $\overline{C} > 0 $ such that
\[
	  \sup _{v\in I _n (x)} | W _\ttheta (x) - W _\ttheta (v) |\leqslant\overline{C}\,\lambda ^n (x)
\]
for all $ x\in\J $, $ n\in\N $ and $\ttheta\in\T ^{\N _0} $.
\item If $ W $ is non-degenerate, then there is a constant $ c > 0 $ such that
\[
	c \,\lambda ^n (x) \leqslant \sup _{v\in\J\cap I _n (x)} | W  (x) - W(v) |
\]
for all $ x\in\J $ and $ n\in\N $.
\item If $ W _\ttheta $ is non-degenerate for almost all $\ttheta $, then there is a measurable function $ c :\T ^{\N _0}\to [ 0,\infty) $ such that
\[
	c (\sigma ^n\ttheta ) \,\lambda ^n (x) \leqslant \sup _{v\in\J\cap  I _n (x)} | W _\ttheta  (x) - W _\ttheta (v) |
\]
for all $ x\in\J $, $ n\in\N $ and $\ttheta\in \T ^{\N _0 } $, where $\sigma :\T ^{\N _0}\to\T ^{\N _0} $ is the left shift operator.
In addition, $\lim _{n\to\infty}\frac{\log c (\sigma ^n\ttheta )}{n} = 0 $ for almost all $\ttheta$.
\end{itemize}
\end{lemma}
\begin{proof}
The proof is provided in Section \ref{sec:proof_Bed_weak}.
\end{proof}

As a corollary of these lemmas, the nowhere-differentiability argument follows.

\begin{proof}[Proof of Lemma \ref{lem:degenerate}]
Assume that $W$ is not of class $C^{1+\alpha}(\T )$.
In \cite[Section 5]{Bedford89}, it is shown in case $\J =\T$ and $\tau ' > 0$ that $W$ is not locally Lipschitz in each fixed point of $\tau$.
In our slightly more general setting, the same statement can be similarly verified.
In particular, we are in the non-degenerate situation.
Let $x\in\J$ be arbitrary.
By the second point of Lemma \ref{lem:Bedford_weak}, for each $n\in\N$ there is a $u_n\in\J\cap I_n(x)$ such that $c \lambda ^n(x)\leqslant |W(x)-W(u_n)| $.
On the other hand, by Lemma \ref{lem:tau_distortion}, $|x -u_n|\leqslant|I_n(x)|\leqslant D/|(\tau ^n)'(x)|$ for all $n\in\N$.
Hence, with $\Delta$ denoting the value on the left hand side of the partial hyperbolicity condition \eqref{eq:partial_hyperbolic}, we have $\lim _{n\to\infty}u_n =x$ but
\[
	\left|
	\frac{W(x)-W(u_n)}{x - u_n}\right|\geqslant\frac{c|(\tau ^n)'(x)|\lambda ^n(x)}{D}\geqslant (c/D) \Delta^n\quad\stackrel{n\to\infty}{\longrightarrow}\infty.
\]
Thus the derivative of $W$ in $x$ does not exist.
\end{proof}

\subsection*{Proofs of Theorems \ref{thm:dim} and \ref{thm:graph_dim_random}}

The proofs consist of the three Lemmas \ref{lem:box_hausdorff_upper}, \ref{lem:box_lower} and \ref{lem:lift}.
All upper bounds for the dimensions in the both theorems are provided in the first lemma, while the lower bounds are shown in several steps as follows.
For the box dimension, the lower estimates follow from the second lemma, whose assumption is satisfied in view of the second and third point of Lemma \ref{lem:Bedford_weak} in cases of Theorems \ref{thm:dim} and \ref{thm:graph_dim_random}, respectively.
In order to obtain the lower bound for the Hausdorff dimension stated in Theorem \ref{thm:graph_dim_random}, we need to apply the third lemma to the equilibrium states $\nu_1$ and $\nu_2$ for the pressures defined in \eqref{eq:s12}, respectively.
From that lemma, we can conclude the following for almost all $\ttheta$:
If $s_1\leqslant s_2$, then the lift of $\nu_1$ on the graph of $W_\ttheta$ has dimension $s_1$, whereas if $s_1>s_2$, the lift of $\nu_2$ has dimension $s_2$.
In view of Lemma \ref{lem:Barreira_cite}, it follows that $\min\{s_1,s_2\}\leqslant\dim _H (\graph W_\ttheta\cap (\J\times\R))$ for almost all $\ttheta$.

\begin{lemma}	\label{lem:box_hausdorff_upper}
For any $\ttheta\in\T^{\N_0}$, we have
\[
	\overline{\dim} _B(\graph W _\ttheta\cap(\J\times\R))\leqslant s _1\quad\mbox{and}\quad\dim _H(\graph W _\ttheta\cap(\J\times\R))\leqslant\min\{s _1 , s _2\}.
\]
\end{lemma}

\begin{proof}
Let $\ttheta$ be fixed.
We first consider the box dimension.
For $r>0$, let $N _r$ be the least number of closed squares with side length $r$ that are needed to cover $\graph W _\ttheta\cap(\J\times\R)$.
Recall that $\overline{\dim} _B(\graph W _\ttheta\cap(\J\times\R))=\limsup _{r\to 0}\frac{\log N _r}{-\log r}$ by definition.
To find the bound, we use Moran covers of $\J$, so let $\mathfrak{U} _r :=\{ I _{n _r (x)}(x):x\in\J\}$, where $n _r (x):=\min\{ n\in\N :|I _n(x)|< r\}$.
Observe that $\J\subseteq\bigcup _{I\in\mathfrak{U} _r}\overline{I} $ and $\delta _0 r\leqslant |I|<r $ for all $I\in\mathfrak{U} _r$, where $\delta _0 >0$ is the constant from Lemma \ref{lem:tau_distortion}.
In view of Lemma \ref{lem:Bedford_weak}, for each $I _n (x)\in\mathfrak{U} _r$, the part of the graph $\graph  W _\ttheta \cap (\overline{ I _n(x)}\times\R )$ can be covered by $\left\lceil\overline{C}\lambda ^n (x)/|I _n (x)|\right\rceil$ closed squares with side length $r$.
In view of the partial hyperbolicity \eqref{eq:partial_hyperbolic} and Lemma \ref{lem:tau_distortion}, this number does not exceed $(\overline{C}+D)\lambda ^n (x)/|I _n (x)|$.
Now, let $\nu _1\in\mathcal{P} _\tau (\J )$ be the equilibrium state for the topological pressure of the definition of $s _1$ in \eqref{eq:s12}, which is a Gibbs measure according to Lemma \ref{lem:eqisGibbs}.
Together with Lemma \ref{lem:tau_distortion}, there is thus a constant $C _1 >0 $ such that
\[
	C _1 ^{-1}\leqslant\frac{\nu _1 (I _n (x))}{|I _n (x)| ^{s_1 -1}\lambda ^n (x)}\leqslant C _1
\]
for all $x\in\J$ and $n\in\N$.
Hence we can conclude
\begin{eqnarray*}
	N _r	\leqslant\quad 	\sum _{I _n (x)\in\mathfrak{U} _r}(\overline{C}+D)\frac{\lambda ^n (x) }{| I _n (x) | } 
	&\leqslant	&	\sum _{I\in\mathfrak{U} _r}  (\overline{C} + D) C _1\,\nu _1 ( I ) | I  | ^{- s_1}\\
	&\leqslant	& (\overline{C} + D) C_1 (\delta _0 r ) ^{- s_1}\sum _{I\in\mathfrak{U} _r}\nu _1 ( I )\quad = (\overline{C} + D) C _1 \delta _0 ^{- s_1}\, r  ^{- s_1} 
\end{eqnarray*}
for all $ r > 0 $.
This finishes the proof for the box dimension.

We turn to the Hausdorff dimension.
As $\dim _H(\graph W _\ttheta\cap(\J\times\R))\leqslant\overline{\dim} _B(\graph W _\ttheta\cap(\J\times\R))\leqslant s _1$, it remains to show $\dim _H(\graph W _\ttheta\cap(\J\times\R))\leqslant s _2$.
To this end, we apply a general formula for local Hölder exponent that is stated as Lemma \ref{lem:SJ} in a later section.
Together with Lemma \ref{lem:Bedford_weak}, we have
\begin{equation}
	\hol _{W _\ttheta} ( x ) =\liminf _{n\to\infty}\inf _{u\in I _n (x)}\frac{\log | W _\ttheta (x) - W _\ttheta (u)|}{\log | I _n (x) |} \geqslant\liminf _{n\to\infty}\frac{\log\lambda ^n (x) }{\log  | I _n (x) |} =:\tilde{h} (x)  	\label{eq:hol_geq_holtilde_proof}
\end{equation}
for all $ x\in\J\setminus\nullset $.
We consider the sub-level sets $ E ^< _\alpha :=\{ x\in\J\setminus\nullset :\tilde{h} (x) <\alpha\} $ for $\alpha > 0 $.
The crucial fact is that
\begin{equation}	\label{eq:proof_upper_estimate}
	\dim _H E  ^<  _\alpha\leqslant s _2\,\alpha
\end{equation}
holds for all $\alpha > 0 $.
Postponing its proof, we first demonstrate how the rest of the main proof can be derived from this by means of the following very general inequality:
\[
	\dim _H\{ (x, W _\ttheta (x) ) : x\in E ^< _\alpha\mbox{ and }\hol _{W _\ttheta} (x)\geqslant\beta\}\leqslant\frac{\dim _H E ^< _\alpha}{\beta} 
\]
for any $\alpha ,\beta > 0 $, see \cite[Theorem 1]{Jin11}.
Choose a sufficiently large interval, say, $[\alpha _{\min},\alpha _{\max}+1]$.
Given $ N\in\N $, we partition this interval equally into $N$ subintervals, so let $t _i :=\alpha _{\min}+\frac{i}{N}(1+\alpha _{\max}-\alpha _{\min})$ for $i=0,\ldots ,N$.
Observe that for each $ x\in\J\setminus\nullset $ there is some $ i\in\{ 0 ,\ldots , N-1\} $ such that $ t _i\leqslant\tilde{h} (x) < t _{i+1} $, which together with \eqref{eq:hol_geq_holtilde_proof} implies that $\hol _{W _\ttheta} (x)\geqslant\tilde{h} (x)\geqslant t _i $.
Thus we have $\J\setminus\nullset =\bigcup _{i = 0} ^{N -1}E ^< _{t _{i+1}} \cap\{\hol _{W _\ttheta}\geqslant t _i\}$.
As the countable set $\nullset$ has Hausdorff dimension zero, it follows that
\begin{eqnarray*}
	\dim _H\graph (W _\ttheta\cap (\J\times\R)) 	& = &	\max _{i =0,\ldots , N-1}\dim _H\left\{ (x, W _\ttheta (x) ) : x\in E ^< _{t _{i+1}}\mbox{ and }\hol _{W _\ttheta} (x)\geqslant t _i\right\}\\
	&\leqslant &	\max _{i =0,\ldots , N-1}\frac{\dim _H E ^< _{t _{i+1}}}{t _i}\\
	&\leqslant &	\max _{i =0,\ldots , N-1}\frac{s _2\, t _{i+1}}{t _i}\quad\leqslant s _2\left( 1 +\frac{1}{\alpha _{\min} N}\right).
\end{eqnarray*}
Hence, letting $ N\to\infty $ yields the claimed upper bound $s_2$ for the Hausdorff dimension.

Finally, it remains to show the inequality \eqref{eq:proof_upper_estimate}.
Given $ r > 0 $, we can define
\[
	n _r ( x ) :=\min\left\{  n\in\N : | I _n (x) |\leqslant r\mbox{ and }| I _n (x) | ^\alpha\leqslant\sup _{u\in I_n(x)}\lambda ^n (u) \right\}  
\]
for each $ x\in E ^< _\alpha $.
Then, $\mathfrak{U} _r :=\{ I _{n _r (x)} (x) : x\in E ^< _\alpha\} $ is a family of disjoint intervals such that $ E ^< _\alpha\subseteq\bigcup _{I\in\mathfrak{U} _r}\overline{I} $ and $ | I |\leqslant r $ for all $ I\in\mathfrak{U} _r $.
Let $\nu _2\in\mathcal{P} _\tau (\J ) $ be the equilibrium state for the topological pressure of the definition of $ s _2 $ in \eqref{eq:s12}.
By Lemma \ref{lem:eqisGibbs}, there is a constant $C_2 > 0$ such that
\[
	C _2^{-1}\leqslant\frac{\nu _2 ( I _n (x) ) }{(\lambda ^n(u)) ^{s_2}}\leqslant C _2
\]
for all $u\in I_n(x)$, $ x\in\J $ and $ n\in\N $.
In particular, for each $ I _n (x)\in\mathfrak{U} _r $ we have
\[
	| I _n (x) | ^{ s _2\alpha }\leqslant \sup _{u\in I _n(x)}(\lambda ^n (u) ) ^{s _2}\leqslant C _2\,\nu _2 ( I _n (x) ).
\]
Thus, for any $ d > s_2\alpha $, the $ d $-Hausdorff measure of $ E ^< _\alpha $ is bounded as
\begin{eqnarray*}
	\mathcal{H} ^d ( E ^< _\alpha )	\leqslant\quad 	\sum _{I\in\mathfrak{U} _r} | I | ^d  & = &\sum _{I _n (x)\in\mathfrak{U} _r} | I _n (x) | ^{d - s_2\alpha} | I _n (x) | ^{s_2\alpha}\\
	&\leqslant &	 C _2\, r ^{d - s_2\alpha}\sum _{I\in\mathfrak{U} _r}\nu _2 ( I )\quad\leqslant C _2\, r ^{d - s_2\alpha} .
\end{eqnarray*}
Letting $ r\to 0 $, we obtain $\mathcal{H} ^d ( E ^< _\alpha ) = 0 $, so $\dim _H E ^< _\alpha\leqslant d $.
As $d>s_2\alpha$ was arbitrary, the claim is proved.
\end{proof}

\begin{lemma}	\label{lem:box_lower}
Given $\ttheta$, assume that there is a sequence $ ( c _n ) _{n\in\N}\subset ( 0,\infty ) $ with $\lim _{n\to\infty}\frac{\log c_n}{n} = 0 $ such that
\[
	c _n\,\lambda ^n (x)\leqslant\sup _{v\J\cap\in I _n (x)} | W _\ttheta (v ) - W _\ttheta (x)|
\]
for all $ x\in\J $ and $ n\in\N $.	
Then we have $\underline{\dim} _B(\graph W _\ttheta\cap (\J\times\R))\geqslant s _1 $.
\end{lemma}

\begin{proof}
Let $\nu _1\in\mathcal{P} _\tau (\J ) $ and $ C _1 > 0 $ be as in proof of Lemma \ref{lem:box_hausdorff_upper}, so we have
\[
	C _1 ^{-1}\leqslant\frac{\nu _1 ( I _n (x) ) }{| I _n (x) | ^{s_1 -1}\lambda ^n (x)}\leqslant C _1 
\]
for all $ x\in\J $ and $ n\in\N $.
In addition, we again consider $N_r$ for $r>0$, the least number of closed squares with side length $r$ that are needed to cover $\graph W _\ttheta\cap(\J\times\R)$.
Recall that $\underline{\dim} _B(\graph W _\ttheta\cap(\J\times\R))=\liminf _{r\to 0}\frac{\log N _r}{-\log r}$ by definition.
Now, for an arbitrary $\varepsilon >0$, let $D _N\subseteq\T$ denote the set of those $ x\in\J $ which satisfy
\[
	\max\left\{\left|\frac{\log\lambda ^n ( x )}{n}  -\int\log\lambda\, d\nu _1\right|, \left|\frac{-\log\nu _1 ( I _n ( x ) )}{n} - h _\tau (\nu _1 )\right|, \left|\frac{-\log | I _n (x) |}{n}  -\int\log  |\tau ' |\, d\nu _1\right|\right\}<\varepsilon
\]
for all $n\geqslant N$.
Observe that $\lim _{N\to\infty}\nu _1 ( D _N ) = 1 $ in view of Birkhoff's ergodic theorem and Shannon–McMillan–Breiman theorem together with Lemma \ref{lem:tau_distortion}.
Hence we can choose a number $ N _0\in\N $ so that $\nu _1 ( D _{N_0} ) > 0 $.
Let $\mathfrak{U} _n :=\{ I _n (x) : x\in D _{N_0}\} $ for $ n\geqslant N _0 $.
Clearly, these are coverings of $D _{N_0}$.
Since
\[
	C _1\geqslant\frac{\nu _1 ( I _n (x) ) }{| I _n (x) | ^{s _1 - 1}\lambda ^n (x)}\geqslant\nu ( I _n (x) )\, e ^{ n ( - (1 - s_1)\int\log |\tau ' |\, d\nu _1 -\log\lambda\, d\nu _1 -s_1\varepsilon) } 
\]
for each $ I _n(x)\in\mathfrak{U} _n$, we have
\begin{eqnarray*}
	\#\mathfrak{U} _n 	&\geqslant &	C _1 ^{-1}\sum _{I\in\mathfrak{U} _n} e ^{n ( - (1 - s_1)\int\log |\tau ' |\, d\nu _1 -\log\lambda\, d\nu _1 -s_1\varepsilon)}\nu _1 ( I )\\
	&\geqslant &	 C _1 ^{-1}\nu _1 ( D _{N_0} )\, e ^{n ( - (1 - s_1)\int\log |\tau ' |\, d\nu _1 -\log\lambda\, d\nu _1 -s_1\varepsilon)} 
\end{eqnarray*}
for all $ n\in\N $.
On the other hand, in view of the choice of $(c_n)_n$, the hight of $\graph  W _\ttheta\cap ( I\times\R ) $ for $ I\in\mathfrak{U} _n $ is at least $ e ^{n (\int\log\lambda\, d\nu _1 -\varepsilon ) +\log c _n} $.
Let $ r _n := e ^{n ( -\int\log |\tau ' |\, d\nu _1 -\varepsilon) }$ so that $| I |\geqslant r _n $ for all $I\in\mathfrak{U} _n $.
Since $W_\ttheta$ is continuous, we have
\[
	N _{r _n}\geqslant\frac{\#\mathfrak{U} _n}{2}\,\frac{e ^{n (\int\log\lambda\, d\nu _1 -\varepsilon )+\log c _n}}{r _n}\geqslant (2C _1) ^{-1}\nu _1 ( D _{N_0} )\, e ^{n ( s _1\int\log |\tau ' |\, d\nu _1 - s _1\varepsilon )+\log c _n} .
\]
Consequently,
\[
	\liminf _{r\to 0}\frac{\log N _r}{-\log r } =\liminf _{n\to\infty}\frac{\log N _{r _n}}{-\log r _n}\geqslant\frac{s _1\int\log |\tau ' |\, d\nu _1 - s _1\varepsilon}{\int\log |\tau ' |\, d\nu _1  +\varepsilon},
\]
where the above equality is due to monotonicity of the sequence $ ( N _r ) _{r > 0} $ and the fact $\lim _{n\to\infty}\frac{\log r _n}{\log r _{n+1}} = 1 $.
Letting $\varepsilon\to 0 $ finishes the proof.
\end{proof}

\begin{lemma}	\label{lem:lift}
Suppose that $\ttheta\in\T ^{\N _0} $ is an independent and identically uniformly distributed random sequence on $\T$ and that $\tau _{|\J}$ is expansive.
Furthermore, suppose that $ g $ satisfies the critical point hypothesis.
Then, for any Gibbs measure $\nu\in\mathcal{P} _\tau (\J ) $, we have
\[
	\dim _H\mu _\ttheta  =\min\left\{ \dim _H\nu + 1 +\frac{\int\log\lambda\, d\nu}{\int\log |\tau ' |\, d\nu} ,\frac{h _\tau (\nu )}{\int\log |\tau ' |\, d\nu}\right\} 
\]
for almost all $\ttheta$.
\end{lemma}
\begin{proof}
The proof will be given in Section \ref{sec:lift}.
\end{proof}

\begin{remark}
In proof of \cite[Proposition 2.3]{Moss12}, in order to establish a lower bound of $\dim _H\graph W _\ttheta  $ in case $\J =\T $ and $\tau ' > 0 $, the authors actually proved the following:
The Hausdorff dimension of the lift of a Gibbs measure $\nu _\phi $ for a potential $\phi $ is almost surely bounded from below by the number $s$ determined by the equation
\[
	P ( (s-1)\log |\tau' | + (\phi - P(\phi) ) -\log\lambda  ) = 0  .
\]
This approach yields the sharp lower bound of the dimension of the lifted measure for the specific choice $\phi = ( 1 - s )\log |\tau ' | +\log\lambda $, the one they needed.
For any other $\phi $, however, this is rarely the case as pointed out in \cite[Remark 5]{Jin11}.
Indeed, the precise dimension is shown in Lemma \ref{lem:lift}.
\end{remark}

\section{Hölder exponent spectra}
In this section, we study the spectra of the local Hölder exponent of $W_\ttheta$, giving proofs of Theorems \ref{thm:dimE} and \ref{thm:random_spectrum}.

Firstly, we introduce two useful formulas for the local Hölder exponent.
The definition of the exponent for $W_\ttheta$ was given in \eqref{eq:def_holder}, which can be naturally applied for any continuous function $\phi :\T\to\R$, too.
Let $B_r(x):=\{u\in\T :d(x,u)\leqslant r\}$ be closed balls for $x\in\T$ and $r>0$.

\begin{lemma}	\label{lem:SJ}
For any continuous function $\phi : \T\to\R $, we have
\[
	\hol _\phi (x) =\liminf _{r\searrow 0}\inf _{u\in B_r(x)}\frac{\log |\phi (x) -\phi (u)|}{\log r} 
\]
for all $ x\in\T $.

In addition, we have
\[
	\hol _\phi ( x ) =\liminf _{n\to\infty}\inf _{u\in I _n (x)}\frac{\log |\phi (x) -\phi (u)|}{\log | I _n (x) |} 
\]
for all $ x\in\J\setminus\nullset $.
\end{lemma}

\begin{proof}
We start with the first claim.
Fix $ x\in\T $, and let $h_x$ denote the value on the right hand side of the equation.
Let $\varepsilon >0$ be arbitrary.
Then we can choose a zero sequence $(r _k)_k\subset (0,1)$ so that
\[
 	\inf _{u\in B _{r_k} (x)}\frac{\log |\phi (x) -\phi (u)|}{\log r _k} <h_x+\varepsilon .
\]
In view of continuity, for each $k$, there is a $u_k\in B _{r_k}(x)$ such that $|\phi (x)-\phi (u_k)|>r_k^{h _x+\varepsilon}\geqslant d(x,u_k)^{h _x+\varepsilon}$.
Therefore,
\[
	\sup _{u\in B_{r_k}(x)}\frac{|\phi(x) -\phi(u)|}{d(x,u)^{h_x+2\varepsilon}}\geqslant\frac{|\phi (x)-\phi (u_k)|}{d(x,u_k)^{h_x+2\varepsilon}}\geqslant d(x,u_k)^{-\varepsilon}\geqslant r_k^{-\varepsilon}.
\]
Thus, $\inf _{r>0}\sup _{u\in B_r(x)}\frac{|\phi (x) -\phi (u)|}{d(x,u)^{h_x+2\varepsilon}} = \lim _{k\to\infty}\sup _{u\in B_{r_k}(x)}\frac{|\phi(x) -\phi(u)|}{d(x,u)^{h_x+2\varepsilon}}=\infty$, so $\hol _\phi(x) \leqslant h_x+2\varepsilon$.
On the other hand, there is some $r _0\in (0,1)$ such that $\sup _{u\in B_r(x)}|\phi (x)-\phi (u)|<r^{h_x-\varepsilon}$ for all $r\in (0,r_0)$.
Hence we have $\hol _\phi (x)\geqslant h_x-\varepsilon$, since
\[
	\inf _{r>0}\sup _{u\in B_r(x)}\frac{|\phi (x) -\phi (u)|}{d(x,u)^{h _x-\varepsilon}}\leqslant\sup _{u\in B_{r _0}(x)}\frac{\sup _{v\in B_{d(x,u)}(x)}|\phi (x) -\phi (v)|}{d(x,u)^{h_x-\varepsilon}}\leqslant 1.
\]
Letting $\varepsilon\to 0$ finishes the proof of the first claim.

Let $I_L$ and $I_R$ be the leftmost and rightmost intervals defined as follows.
$I_L:=I_0$ if $0\in\overline{I _0}$, and otherwise $I_L$ is the closed interval from $0$ to the left endpoint of $I_0$. Similarly, $I_R:=I_{\ell -1}$ if $1\in\overline{I _{\ell -1}}$, and otherwise from the right endpoint of $I _{\ell -1}$ to $1$.
Recall that the distance between a set and a point is defined as $\dist ( A , x ) :=\inf\{ d(x,u) : u\in A\} $ for $ A\subseteq\T $ and $ x\in\T$.
With the constant $D>0$ from Lemma \ref{lem:tau_distortion}, let $\delta := D ^{-1}\min\{ |I _L| , |I _R|\}$, so that holds the implication:
\[
	\dist (\partial ( I _n (x) ) ,\{x\} ) <\delta | I _n (x) |\quad\Longrightarrow\quad\tau ^n x\in I _L\cup I _R.
\]
To see this, let $u\in\partial I_n(x)$ satisfy $d(u,x)<\delta |I_n(x)|$. If $x\in\{0,1\}$, the claim is trivial.
Otherwise, by the mean value theorem,
\[
	d(\tau ^n x,\tau ^n u)\leqslant\sup _{w\in I_n(x)}|(\tau ^n)'(w)|\cdot d(x,u)<\frac{D}{|I_n(x)|}\cdot \delta |I_n(x)|\leqslant \min\{|I_L|,|I_R|\}.
\]
As $\tau ^n u\in\{0,1\}$, this means $\tau ^n x\in I _L\cup I _R$, so the claimed implication is shown.

Using the above constant $\delta$, let
\[
	\J _0 :=\bigcap _{N\in\mathbb{N}}\bigcup _{ n\geqslant N}\left\{ x\in\J :\dist (\partial ( I _n (x) ) ,\{x\} )\geqslant\delta | I _n (x) |\right\} .
\]
We claim $\J\setminus\nullset\subseteq\J _0 $.
Observe that it is equivalent to show $\J\setminus\J _0\subseteq\nullset $, so let $ x\in\J\setminus\J _0 $.
Then, there is some $ N\in\N $ such that 
\[
	\dist (\partial ( I _n (x) ) ,\{x\} ) <\delta | I _n (x) | 
\]
for all $ n\geqslant N $.
As shown above, this means $\tau ^n x\in I _L\cup I _R $ for all $n\geqslant N$, which is only possible when $\tau ^{N+1} x\in\{ 0, 1\} $.
Thus, $ x\in\nullset $.

We proceed the proof of the lemma.
Let $x\in\J\setminus\nullset$.
Since $ x\in\J _0 $, the inequalities
\[
	\inf _{u\in B _{ | I _n (x) | }(x)}\frac{\log |\phi (x) -\phi (u) |}{\log | I _n (x) |}\leqslant\inf _{u\in I _n (x)}\frac{\log |\phi (x) -\phi (u) |}{\log | I _n (x) |}\leqslant\inf _{ u\in B _{\delta | I _n (x) | } (x)}\frac{\log |\phi (x) -\phi (u) |}{\log | I _n (x) |}
\]
hold for infinitely many $ n\in\mathbb{N}$.
In addition, observe that the first and thirds expressions have the same limits inferiors as $n\to\infty$, so the value must coincide with the limit inferior of the middle one.
In view of $ | I _n (x) |\searrow 0 $ and $\lim _{n\to\infty}\frac{\log | I _n (x) |}{\log | I _{n+1} (x) |} = 1 $, the value can be calculated as
\[
	\liminf _{n\to\infty}\inf _{u\in  B _{| I _n (x) |} (x)}\frac{\log |\phi (x) -\phi (u) |}{\log |I _n (x)|}	= \liminf _{r\searrow 0}\inf _{u\in B _r (x)}\frac{\log |\phi (x) -\phi (u) |}{\log r} = \hol _\phi (x).
\]
\end{proof}

The following is an immediate corollary of the preceding lemma with the specific choice $\phi=W_\ttheta$, in view of Lemmas \ref{lem:tau_distortion} and \ref{lem:Bedford_weak}.

\begin{lemma}	\label{lem:hol_symbolic}
If $ W $ is non-degenerate, we have
\[
	\hol _W ( x ) =\liminf _{n\to\infty}\frac{-\log\lambda ^n (x) }{\log  |\tau ' | ^n} 
\]
for all $ x\in\J\setminus\nullset $.

In addition, if $ W _\ttheta $ is non-degenerate for almost all $\ttheta $, then the same formula for $W_\ttheta$ is true for almost all $\ttheta$.
\end{lemma}

The next lemma is a collection of several fundamental facts from thermodynamic formalisms.
Let us consider the Hausdorff spectrum of the following (sub-, sup-) level sets:
\begin{equation*}
	S _\alpha :=\left\{ x\in\J :\liminf _{n\to\infty }\frac{-\log\lambda ^n (x)}{\log |\tau ' | ^n  (x)} =\alpha\right\} ,
\end{equation*}
\[
	S _\alpha ^\leqslant :=\left\{ x\in\J :\liminf _{n\to\infty }\frac{-\log\lambda ^n (x)}{\log |\tau ' | ^n (x)}\leqslant\alpha\right\}\quad\mbox{and}\quad S _\alpha ^\geqslant :=\left\{ x\in\J :\liminf _{n\to\infty }\frac{-\log\lambda ^n (x)}{\log |\tau ' | ^n (x)}\geqslant\alpha\right\} .
\]
For the details of the following lemma, consult \cite{Pesin97} or \cite{Barreira08}.
Note that the function $\mathcal{D}:\A\to[0,1]$ as well as the related constants were introduced immediately before Theorem \ref{thm:dimE}.

\begin{lemma}	\label{lem:wellknown}
If $\A\neq\emptyset $, we have
\[ 
	\mathcal{D} (\alpha ) =\dim _H S _\alpha = 
	\begin{cases}
		\dim _H  S _\alpha  ^\leqslant &	\mbox{for }\alpha\in(\alpha _{\min},\alpha _c)\\
		\dim _H S _\alpha  ^\geqslant &	\mbox{for }\alpha\in [\alpha _c,\alpha _{\max})
	\end{cases}.
\]
Moreover, there is a Gibbs measure $\nu _\alpha\in\mathcal{P} _\tau (\J ) $ such that $\nu _\alpha ( S _\alpha ) = 1 $ and $	\dim _H\nu _\alpha =\dim _H S _\alpha$.
\end{lemma}

It is now immediate to determine the Hausdorff spectrum $\alpha\mapsto\dim _H E _\alpha $.

\begin{proof}[{Proof of Theorems \ref{thm:dimE}}]
By Lemma \ref{lem:hol_symbolic} we have $ E _\alpha\bigtriangleup S _\alpha\subseteq\nullset $, and thus $\dim _H E _\alpha =\dim _H S _\alpha $ for all $\alpha\in\R $.
Therefore, in case $\A\neq\emptyset $, the assertion of the theorem follows from Lemma \ref{lem:wellknown}.
Moreover, in case $\A =\emptyset$, we obtain $\hol _W (x) =\alpha _c $ for all $ x\in\J\setminus\nullset $ from that lemma.
Assuming $\A=\emptyset $, it remains to show that $\hol _W (x)=\alpha _c$ for all $ x\in\J\cap\nullset $.
Observe that the above assumption is equivalent to that $(\alpha _c\log |\tau ' | +\log\lambda)_{|\J} $ is cohomologous to $ 0 $, i.e. it is equal to $\phi\circ\tau -\phi $ for some bounded function $\phi :\J\to\R $.\footnote{To see this, observe: $\A =\emptyset$ $\Leftrightarrow$ $-A_q= q\alpha_c$ for $\forall q\in\R$ $\Leftrightarrow$ $P_\J (q(\alpha _c \log |\tau '|+\log \lambda ))=0$ for $\forall q\in\R$.}
In particular, in view of Lemma \ref{lem:tau_distortion}, there is a constant $ C _1 > 0 $ such that $ C _1 ^{-1}\leqslant | I _n (u) | ^{-\alpha _c}\lambda ^n (u)\leqslant C _1 $ for all $ u\in\J $ and $ n\in\N $.
Hence, by Lemma \ref{lem:Bedford_weak},
\[
	\sup _{u\in I_n(x)}|W(u)-W(x)|\geqslant c\lambda ^n(x)\geqslant c C_1^{-1}|I_n(x)|^{\alpha _c}.
\]
for all $x\in\J$ and $n\in\N$.
On the other hand, by the first formula of Lemma \ref{lem:SJ} together with the fact $\lim _{n\to\infty}\frac{\log |I_n(x)|}{\log |I_{n+1}(x)|}=1$, we have
\[
	\hol _W(x)=\liminf _{n\to\infty}\frac{\log \, \sup _{u\in B_{|I_n(x)|}(x)}|W(x)-W(u)|}{\log |I_n(x)|}
\]
for all $x\in\T$, where $B_r(x):=\{u\in\T :|x-u|\leqslant r\}$ for $x\in\T$ and $r>0$.
Since $I_n(x)\subseteq B_{|I_n(x)|}(x)$, it follows immediately that $\hol _W(x)\leqslant\alpha _c$ for all $x\in\J$, and especially for all $x\in\J\cap\nullset$.
To show the inverse inequality, we  distinguish three cases.

Firstly, if $\{0,1\}\cap\J =\emptyset$, then $\J\cap\nullset =\emptyset$, so there is nothing to prove.

Secondly, assume $\{0,1\}\subseteq\J$.
Given $x\in\J\cap\nullset$ and $n\in\N$, there are $ m _n\in\N $ and $ x _n\in\J $ so that
\[
	\overline{I_n(x)}\cap\overline{I _{m _n} (x_n)}=\{x\}\quad\mbox{and}\quad	\delta _0| I _n (x) |\leqslant | I _{m _n} (x_n) |\leqslant | I _n (x) | ,
\]
where $\delta _0 > 0 $ is the constant from Lemma \ref{lem:tau_distortion}.
Thus we have
\[
	B_{|I_n(x)|}(x)\subseteq\overline{I _n ( x )\cup I _{m _n} ( x _n )}.
\]
In addition, by Lemma \ref{lem:Bedford_weak},
\[
	\sup _{u\in\overline{I _n ( x )\cup I _{m _n} ( x _n )}}|W(u)-W(x)|\leqslant\overline{C}(\lambda ^n(x)+\lambda ^{m_n}(x_n))\leqslant\overline{C}C_1(1+\delta _0^{-\alpha _c})|I_n(x)|^{\alpha _c}.
\]
Hence we can choose a constant $C_2>0$ so that
\[
	\sup _{u\in B _{|I_n(x)|}(x)}|W(u)-W(x)|\leqslant C_2\, |I_n(x)|^{\alpha _c}
\]
for all $x\in\J\cap\nullset $ and $n\in\N$, which in terns means that $\hol _{W}(x)\geqslant\alpha _c$ for all $x\in\J\cap\nullset$.

Finally, assume that only one of $0$ and $1$ belongs to $\J$.
We define the right ball $B_r^+(x)\subseteq \T\approx\R/\Z$ as the canonically projected image of $[x,x+r]\subseteq\R$, and similarly, the left ball $B_r^-(x)$.
Let $x\in\J\cap\nullset $.
For sufficiently large $n$, by the assumption here, the interval $\overline{I_n(x)}$ must correspond with $B^-_{|I_n(x)|}(x)$ or $B^+_{|I_n(x)|}(x)$, say with $B^-_{|I_n(x)|}(x)$.
Let $r_+:=\sup\{r>0:B^+_{r_0}(x)\cap\J=\{x\}\}$ where $\sup\emptyset :=0$.
If $r_+=0$, then there exist $(m_n)_n$ and $(x_n)_n$ as in the second case and the same arguments apply.
Thus it remains the case $r_+>0$.
As $W$ is $C^1$ on $(B _{r_+}^+(x))^\circ$ by Lemma \ref{lem:smooth_part}, there is a number $C(x)\geqslant C_2$ in dependence of $x$ such that
\[
	\sup _{u\in B_r^+(x)}|W(u)-W(x)|\leqslant C(x)\, r
\]
for all $r\in (0,1)$.
From $ B_{|I_n(x)|}(x)=\overline{I _n(x)}\cup B_{|I_n(x)|}^+(x)$ follows
\[
	\sup _{u\in B_{|I_n(x)|}(x)}|W(u)-W(x)|=\sup _{u\in I_n(x)\cup B^+_{|I_n(x)|}(x)}|W(u)-W(x)|\leqslant (\overline{C} + C(x))\, |I_n(x)|^{\alpha _c}
\]
for all $n\in\N$, which implies $\hol _W(x)\geqslant \alpha _c$.
\end{proof}

We turn to the lifted spectrum $\alpha\mapsto\dim _H(\graph W_\ttheta\cap (E _{\ttheta, \alpha}\times\R))$ for almost all $\ttheta$.

\begin{proof}[Proof of Theorem \ref{thm:random_spectrum}]
As its upper bound is already provided in Lemma \ref{lem:Jin}, we only discuss the lower estimate.
uniformly distributed random sequence on $\T$.
Assume that $\ttheta\in\T ^{\N _0} $ is an independent and identically uniformly distributed random sequence on $\T$ and that $\tau _{|\J}$ is expansive.
Let $\mu _{\ttheta ,\alpha }=\nu _\alpha\circ (\mathrm{Id}_\T\times W_\ttheta)^{-1} $ be the lift of the Gibbs measure $\nu_\alpha$ introduced in Lemma \ref{lem:wellknown}, so we have $\mu _{\ttheta ,\alpha } (\graph W_\ttheta\cap (S _\alpha\times\R))=1$.
Furthermore, by Lemma \ref{lem:lift} and Theorem \ref{thm:dimE}, for each $\alpha\in\A$, we have almost surely
\begin{eqnarray*}
	\dim _H(\graph W_\ttheta\cap (S _\alpha\times\R))&\geqslant &\dim_H(\mu_{\ttheta,\alpha})\\
	& = &	\min\left\{\dim _H\nu _\alpha + 1 +\frac{\int\log\lambda\, d\nu _\alpha}{\int\log |\tau ' |\, d\nu _\alpha} ,\frac{h _\tau (\nu _\alpha )}{\int\log |\tau ' |\, d\nu _\alpha}\right\}\\
	& = &	\min\left\{\D (\alpha ) + 1 -\alpha ,\frac{\D (\alpha )}{\alpha}\right\}	 .
\end{eqnarray*}
Thus there is a set $Z\subseteq \T ^{\N _0}$ of full probability such that
\[
	\dim _H(\graph W_\ttheta\cap (S _\alpha\times\R))\geqslant\min\left\{\D (\alpha ) + 1 -\alpha ,\frac{\D (\alpha )}{\alpha}\right\}=:\tilde{D}(\alpha)
\]
for all $\alpha\in\A\cap\mathbb{Q} $ and $\ttheta\in Z$.
In view of Lemma \ref{lem:hol_symbolic}, by shrinking $Z$ if needed, we may also assume that $S_\alpha\bigtriangleup E_{\ttheta,\alpha}\subseteq\nullset$ for all $\alpha\in\A\cap\mathbb{Q}$ and $\ttheta\in Z$.
Then, by Lemma \ref{lem:wellknown}, we obtain
\begin{eqnarray*}
	\dim _H(\graph W_\ttheta\cap (E _{\ttheta ,\alpha}\times\R))	&=&	\dim _H(\graph W_\ttheta\cap (S _\alpha\times\R)) \\
	&=&	\dim _H(\graph W_\ttheta\cap (S _\alpha ^\leqslant\times\R)) \\
	&\geqslant &	\dim _H(\graph W_\ttheta\cap (S _{\alpha'} ^\leqslant\times\R)) \\
	&= &	\dim _H(\graph W_\ttheta\cap (S _{\alpha'}\times\R)) =\tilde{D} (\alpha')
\end{eqnarray*}
for all $\alpha'\in (\alpha _{\min}, \alpha]\cap\mathbb{Q}$, $\alpha\in (\alpha _{\min} ,\alpha _c]$ and $\ttheta\in Z$.
As $\tilde{D}$ is continuous, it follows that
\[
	\dim _H(\graph W_\ttheta\cap (E _{\ttheta ,\alpha}\times\R))\geqslant\tilde{D}(\alpha )
\]
for all $\alpha\in (\alpha _{\min} ,\alpha _c]$ and $\ttheta\in Z$.
A similar argument works for $\alpha\in [\alpha _c ,\alpha _{\max} ) $.
\end{proof}

\section{Dimension of the randomised lifted measure}	\label{sec:lift}

This section is dedicated to the proof of Lemma \ref{lem:lift}, which gives the value of Hausdorff dimension of the lifts of any Gibbs measures on the graph of the randomised function $W_\ttheta$.
The actual sharp upper bound will be straightforwardly determined in Lemma \ref{lem:upper_bound}, however, to verify the equality is, in spite of the randomisation, harder and of a major interest here.
The latter will be completed in Lemmas \ref{lem:lower1} and \ref{lem:lower2}, as concluded at the end of this section.

We start with a fundamental observation of Gibbs measures.
Recall that the distance between a set and a point is defined as $\dist ( A , x ) :=\inf\{ d(x,v) : v\in A\} $ for $ A\subseteq\T $ and $ x\in\T$.

\begin{lemma}	\label{lem:gibbs_props}
Let $\nu\in\mathcal{P} _\tau (\J)$ be a Gibbs measure.
Then,
\[
	\lim _{n\to\infty}\frac{\log\dist (\partial I _n (x) , x ) }{\log | I _n (x) |} = 1
\]
for $\nu $-a.a. $ x\in\J$.
In addition, $\dim _H\nu =\frac{h _\tau (\nu ) }{\log|\tau '| d\nu}$.
\end{lemma}

\begin{proof}
First, we show that $\lim _{n\to\infty}\nu (\rho _i^n(\T))=0$ converges exponentially fast for each $i\in\Sigma _{\ell}$.
Since $x_i:=\lim _{k\to\infty}\rho _i^k(0)$ is a fixed point of $\tau$, we have 
\[
	C_\phi^{-1}e^{n(\phi(x_i)+P_\phi)}\leqslant\nu (\rho _i^n(\T))\leqslant C_\phi e^{n(\phi(x_i)+P_\phi)}
\]
for all $n\in\N$, where $ C _\phi $, $P_\phi$ and $\phi $ are as in the definition of Gibbs measure.
As $\phi(x_i)+P_\phi<0$ follows from the left inequality, the right one implies the exponential decay.

Now, we turn to the first claim of the lemma.
Given an arbitrary $\alpha\in ( 0 , 1 ) $, let $ M _n :=\lfloor\alpha n\rfloor $ and
\[
	E _n :=	\tau ^{-n}\left(\rho _0 ^{M_n}(\T)\cup\rho _{\ell -1} ^{M_n}(\T)\right) 
\]
for $ n\in\mathbb{N} $.
Since $\tau ^nx\in\J\cap\left(\rho _0 ^{M_n}(\T)\cup\rho _{\ell -1} ^{M_n}(\T)\right)$, we have
\[
d(0,\tau ^nx)\geqslant\min\{|\rho _0 ^{M_n}(\T)|,|\rho _{\ell -1} ^{M_n}(\T)|\}\geqslant (\inf|1/\tau'|) ^{M_n}.
\]
Furthermore, by the mean value theorem and Lemma \ref{lem:tau_distortion},
\[
	\dist (\partial I _n (x),x)\geqslant D^{-1}d(0,\tau ^nx) \inf _{u\in\T}|\rho _{[x]_n}'(u)|\geqslant D^{-2} \dist (\J, 0) |I _n (x)|\geqslant D^{-2} (\inf |1/\tau'|) ^{M_n}|I_n(x)|
\]
for all $n\in\N$ and $x\in\J$.
Thus, by Lemma \ref{lem:tau_distortion}, we obtain the bound
\begin{equation*}
	\limsup _{n\to\infty}\frac{\log\dist (\partial I _n (x) , x ) }{\log | I _n (x) |}\leqslant 1 +\lim _{n\to\infty} \frac{M _n}{\log | I _n (x) |}\log (\inf |1 /\tau' |) =	1 +\alpha\frac{\log (\inf |1 /\tau'| )}{\log (\sup _\J |1/\tau '|)}
\end{equation*}
for all $ x\in\J\setminus\bigcap _{N\in\N}\bigcup _{n\geqslant N} E _n $.
On the other hand, since
\[
	\nu ( E _n ) =\left(\rho _0^n(\T)\cup\rho _{\ell -1}^n(\T)\right)\leqslant\left(\rho _0^n(\T)\right)+\nu\left(\rho _{\ell -1}^n(\T)\right)
\]
decays exponentially as $n\to\infty$, $\nu\left(\bigcap _{N\in\N}\bigcup _{n\geqslant N} E _n\right) = 0$ follows from Borel–Cantelli lemma.
Thus the bound holds for $\nu $-a.a. $x\in\J$, independently of the choice of $\alpha\in(0,1)$.
Letting $\alpha\searrow 0 $ finishes the proof of this part.

Now, we turn to the second claim on $\dim_H\nu$.
Since $\lim _{n\to\infty}\frac{\log | I _n (x) |}{\log | I _{n+1} (x) |} = 1$ by Lemma \ref{lem:tau_distortion}, and since $B _{\dist(\partial I _n (x) , x )} ( x )\subseteq  I _n ( x )\subseteq B _{| I _n (x) | } (x)$,
it follows from the first claim that
\[
	\underline{d} _\nu (x):=\liminf _{r\to 0}\frac{\log\nu\left( B _r(x)\right)}{\log r}=	\liminf _{n\to\infty}\frac{\log\nu ( B _{|I_n(x)|} (x)) }{\log |I_n(x)|}=\liminf _{n\to\infty}\frac{\log\nu ( I _n (x) ) }{\log | I _n (x) |}=\frac{h_\tau(\nu)}{\int \log |\tau'|\,d\nu}
\]
for $\nu$-a.a. $x\in\J$, where the last equality is due to Shannon–McMillan–Breiman theorem and Birkhoff's ergodic theorem together with Lemma \ref{lem:tau_distortion}.
\end{proof}

\begin{lemma}[Upper bound for $\dim _H \mu _\ttheta$]	\label{lem:upper_bound}
Let $\nu\in\mathcal{P} _\tau (\J ) $ be a Gibbs measure.
Then we have
\[
	\dim _H\mu _\ttheta\leqslant	\min\left\{\dim _H\nu + 1 +\frac{\int\log\lambda\, d\nu}{\int\log |\tau '|\, d\nu} ,\frac{h _\tau (\nu )}{\int\log |\tau '|\, d\nu}\right\} 
\]
for all $\ttheta\in\T ^{\N _0} $. 
\end{lemma}

\begin{proof}
Let $\ttheta\in \T ^{\N _0 } $ be fixed.
By Lemma \ref{lem:Bedford_weak}, there is a constant $\bar{C} > 0 $ so that
\begin{equation}
	\sup _{v\in\J\cap I _n (x)} | W _\ttheta (x) - W _\ttheta (v) |\leqslant\bar{C}\,\lambda ^n (x) 	\label{eq:Moss_proof}
\end{equation}
for all $x\in\J$ and $ n\in\N $.
Given an arbitrary $\varepsilon\in ( 0 , e ^{-\int\log |\tau '| d\nu} ) $, let $ D _N\subseteq\J$ denote the set of those $x\in\J$ which satisfies
\[
\max \left\{
\left|\frac{\log\lambda ^n ( x )}{n}  -\int\log\lambda\, d\nu\right| ,
\left|\frac{-\log\nu ( I _n ( x ) )}{n}  - h _\tau (\nu )\right| ,
\left|\frac{-\log | I _n (x) |}{n}  -\int\log |\tau '|\, d\nu\right| 
\right\}<\varepsilon 
\]
for all $ n\geqslant N $.
Observe that $\nu\left(\bigcup _N D _N\right) = 1 $ by Birkhoff's ergodic theorem, Shannon-McMillan-Breiman theorem and Lemma \ref{lem:tau_distortion}.
Furthermore, let $ G _N :=\left\{ ( x , W _\ttheta (x) ) : x\in D _N\right\} $ and $\mathcal{U} _n ^N :=\{ I _n (x) : x\in D _N\} $ for $n\geqslant N$.
Since $	e ^{ n\left( - h _\tau (\nu ) -\varepsilon\right)} <\nu ( I )$
for each $I\in\mathcal{U} ^N _n$, $\mathcal{U} ^N _n $ contains at most $e ^{n\left( h _\tau (\nu ) +\varepsilon\right)}$ intervals.
On the other hand, in view of \eqref{eq:Moss_proof}, for each $ I _n (x)\in\mathcal{U} ^N _n $ we can cover $ G _N\cap\left( I _n (x)\times\mathbb{R}\right)$ by $\left\lceil\overline{C}\,\frac{\lambda ^n (x ) }{ | I _n ( x ) | }\right\rceil $ cubes with edge length $ | I _n ( x ) | $.
Furthermore, in view of the definition of $D_N$, this number is bounded as
\[
	\left\lceil\overline{C}\,\frac{\lambda ^n (x ) }{ | I _n ( x ) | }\right\rceil\leqslant	\overline{C} \, e^{n\left(\int\log |\tau '|\, d\nu +\int\log\lambda\, d\nu + 2\varepsilon\right)} + 1.
\]
Let $\tilde{\mathcal{U}} ^N _n $ be the set of all those cubes over $x\in D_N$.
Observe that $\tilde{\mathcal{U}}^N _N$ is a covering of $G_N$ since $\mathcal{U} ^N _n$ is one of $D_N$.
In addition, using the definition of $D_N$, we can see
\begin{eqnarray*}
	&&	\sum _{Q\in\tilde{\mathcal{U}} ^N _n}\mbox{edge-length} ( Q ) ^s	\\
	&\leqslant	&	\sum _{I\in\mathcal{U} ^N _n}\left(\overline{C}\, e^{n\left(\int\log |\tau '|\, d\nu +\int\log\lambda\, d\nu + 2\varepsilon\right)} + 1\right)\cdot | I  | ^s\\
	& = &	\sum _{I\in\mathcal{U} ^N _n}\left(\overline{C}\, e^{n\left(\int\log |\tau '|\, d\nu +\int\log\lambda\, d\nu  +\frac{s\log | I |}{n} + 2\varepsilon\right)} + | I | ^s\right)\\
		&\leqslant &	\sum _{I\in\mathcal{U} ^N _n}\left(\overline{C}\, e ^{ n\left(\int\log |\tau '|\, d\nu  +\int\log\lambda\, d\nu  + s\left( -\int\log |\tau '|\, d\nu +\varepsilon\right) + 2\varepsilon\right) } +  e ^{s n ( -\int\log |\tau '|\, d\nu +\varepsilon ) }\right)\\
			&\leqslant &	  e ^{ n ( h _\tau (\nu ) +\varepsilon ) }\cdot\left(\overline{C}\cdot e ^{ n\left(\int\log |\tau '|\, d\nu +\int\log\lambda\, d\nu  + s\left( -\int\log |\tau '|\, d\nu +\varepsilon\right) + 2\varepsilon\right) } +  e ^{s n ( -\int\log |\tau '|\, d\nu +\varepsilon ) }\right)\\
	& = &\overline{C}\cdot e ^{ n\left(  h _\tau (\nu )  + ( 1 - s )\int\log |\tau '|\, d\nu   +\int\log\lambda\, d\nu + (3 + s )\varepsilon\right)  }	 + e ^{ n\left( h _\tau (\nu ) - s\int\log |\tau '|\, d\nu + ( 1 + s)\varepsilon\right) } 	\quad\stackrel{n\to\infty}{\longrightarrow} 0 
\end{eqnarray*}
for all $
	s > f (\varepsilon ) : =\max\left\{\frac{ h _\tau (\nu ) +\int\log\lambda\, d\nu +\int\log |\tau '|\, d\nu  + 3\varepsilon }{\int\log |\tau '|\, d\nu -\varepsilon} ,\frac{h _\tau (\nu )  +\varepsilon}{\int\log |\tau '|\, d\nu -\varepsilon}\right\}$.
Thus $\dim _H ( G _N )\leqslant f (\varepsilon ) $.
Furthermore, as Hausdorff dimension is $\sigma$-stable, we have
\[
	\dim _H\left(\bigcup _N G _N\right) =\sup _N\dim _H ( G _N )\leqslant f (\varepsilon ) .
\]
On the other hand, $\mu _\ttheta\left(\bigcup _N G _N\right)=\nu\left(\bigcup _N D _N\right) = 1 $, so we have $\dim _H (\mu _\ttheta )	\leqslant\dim _H\left(\bigcup _N G _N\right)\leqslant f(\varepsilon)$.
Letting $\varepsilon\to 0$, we obtain
\[
	\dim _H (\mu _\ttheta )\leqslant\lim _{\varepsilon\to 0}	f (\varepsilon	) =\max\left\{ 1 +\frac{  h _\tau (\nu) +\int\log\lambda\, d\nu}{\int\log |\tau '|\, d\nu} ,\frac{h _\tau (\nu )}{\int\log |\tau '|\, d\nu}\right\} =  1 +\frac{  h _\tau (\nu) +\int\log\lambda\, d\nu}{\int\log |\tau '|\, d\nu},
\]
where the last equality is due to \eqref{eq:partial_hyperbolic}.
In view of the formula for $\dim _H\nu$ provided in Lemma \ref{lem:gibbs_props}, it is shown that that the left value of the claimed maximum is a correct upper bound.

In order to derive the other bound, we consider the following rough estimate.
Observe that for any $\beta\in (0,1) $ and $ C > 0 $ we have
\[
	\underline{d} _{\mu _\ttheta} ( x , W _\ttheta (x) ) =\liminf _{n\to\infty}\frac{\log\nu\left\{ v\in\T :  | x - v |\leqslant\beta ^n ,\; | W _\ttheta ( x ) - W _\ttheta (v) |\leqslant C\;\beta ^n\right\}}{\log\beta ^n}
\]
for $\nu $-a.a. $ x $, which is a slight modification of \cite[Proposition 2.1.4]{Barreira08}.
Given an arbitrary $\varepsilon\in ( 0 , e ^{\int\log\lambda\, d\nu } ) $, we consider the same set $ D _N\subseteq\J $, $ N\in\mathbb{N} $ as above.
Observe that
\[
	I _n (x)\subseteq\left\{ v\in [0,1] : | v - x |\leqslant ( e ^{\varepsilon }\tilde{\lambda} ) ^n ,\, | W _\ttheta (v) - W _\ttheta (x) |\leqslant\overline{C}\, ( e ^{\varepsilon }\tilde{\lambda} ) ^n \right\}  
\]
for all $x\in D _N $ and $ n\geqslant N$, where $\tilde{\lambda} := e ^{\int\log\lambda\, d\nu}$.
By choosing $\beta =  e ^{\varepsilon }\tilde{\lambda} $ and $ C =\overline{C} $ in the above formula for $\underline{d} _{\mu _\ttheta}$, we have
\begin{eqnarray*}
	\underline{d} _{\mu _\ttheta} ( x , W _\ttheta (x) )	& = &	\liminf _{n\to\infty}\frac{\log\nu\left\{ v\in [0,1] : | x - v |\leqslant ( e ^{\varepsilon }\tilde{\lambda} ) ^n  ,\, | W _\ttheta (v) - W _\ttheta (x) |\leqslant\overline{C}\, ( e ^{\varepsilon }\tilde{\lambda} ) ^n\right\} }{\log  ( e ^{\varepsilon }\tilde{\lambda} ) ^n  }\\
	&\leqslant &	\liminf _{n\to\infty}\frac{\log\nu ( I _n (x) )}{\log  ( e ^{\varepsilon }\tilde{\lambda} ) ^n }\quad\leqslant\frac{h _\tau (\nu) + \varepsilon}{-\int\log\lambda\, d\nu -\varepsilon} 
\end{eqnarray*}
for all $ x\in D _N  $ and $ N\in\mathbb{N} $.
In view of Lemma \ref{lem:Barreira_cite}, letting $\varepsilon\to 0 $ finishes the proof.
\end{proof}

We turn to the lower bound for $\dim _H\mu _\ttheta$.
From now on, we always assume that $\nu\in\mathcal{P} _\tau (\J ) $ is a Gibbs measure and that $ g $ satisfies the critical point hypothesis and that $\tau _{|\J}$ is expansive.
We can choose the expansivity constant $\delexp >0$ so small that\footnote{Expansivity assumption excludes possible V-shaped branches over $\J$, i.e. there is no $x\in\J$ in which the left and right derivatives of $\tau :\R/\Z\to\I$ have opposite signs.} for any $x, v\in\J$
\[
		d _2 ( x , v ) \leqslant\delexp\quad\mbox{implies}\quad v\in I _1 ( x ),
\]
where $d _n ( x , v ) :=\max\{ d (\tau ^j x,\tau ^j v ) : j = 0 ,\ldots , n-1\} $ denotes the $n$-the Bowen metric for $n\in\N$.
In addition, let $ B _{n, r} ( x ) :=\{ u\in\J : d _n (x, u)\leqslant r\} $ denote the ball with respect to $ d _n $ with centre $ x\in\J $ and radius $ r > 0 $.
$B_r(x):=B_{0,r}(x)$.

For $\varepsilon\in ( 0 ,1 )$ and $N\in\N$, let $\mathcal{C} _{\varepsilon ,N}\subseteq\J$ be the set of those $ x $ for which hold
\begin{itemize}
\item $ B _{e ^{n (-\int\log|\tau '| d\nu -\varepsilon)} } ( x )\subseteq B _{n,\delexp} ( x )\subseteq B _{e ^ {n (-\int\log |\tau '| d\nu +\varepsilon)} } ( x ) $,
\item $  e ^{n (\int\log\lambda\, d\nu -\varepsilon ) }\leqslant\lambda ^n (x)\leqslant  e ^{n (\int\log\lambda\, d\nu +\varepsilon ) } $, and
\item $ r ^{\dim _H (\nu) +\varepsilon}\leqslant\nu (B _r(x))\leqslant r ^{\dim _H (\nu) -\varepsilon}  $
\end{itemize}
for all $ n\geqslant N$ and $ r > 0 $.

\begin{lemma}	\label{lem:C_full}
$\lim _{N\to\infty}\nu (\mathcal{C} _{\varepsilon, N}) = 1 $ for all $\varepsilon \in (0,1)$.
\end{lemma}

\begin{proof}
It suffices to show that three conditions are satisfied for all $n\geqslant N_x$ for $\nu$-a.a. $x\in\J$, where $x\mapsto N_x$ is some function.
Moreover, the three conditions can be separately considered, that is, if we know such a $N_x$ for each condition, then their maximum is the one we need.
The existence of such $N_x$ for the second and third conditions is immediate since, by Birkhoff's ergodic theorem and Lemma \ref{lem:gibbs_props},
\[
	\lim _{n\to\infty}\frac{\log\lambda ^n (x)}{n} =\int\log\lambda\, d\nu\quad
\mbox{and}
\quad
	\lim _{r\to 0}\frac{\log\nu ( B _r (x) )}{\log r} =\dim _H\nu 
\]
for $\nu$-a.a. $x$.
Thus we focus on the first condition.
W.l.o.g. assume $N\geqslant 2$ so large that $d(v,w)=|v-w|$ for all $w\in I_N(v)$ and $v\in\J$.
Let $x\in\J$ and $n\geqslant N$.
For $v\in B _{\delexp,n}(x)$, we have $v\in I _{n-2} (x)$, since $\tau ^jv\in I_1(\tau ^jx)$ for $j=0,\ldots n-2$ due to the choice of $\delexp$.
In particular, $d(x,v)\leqslant |I _{n-2} (x)|\leqslant D/|(\tau ^{n-2})'(x)|$ by Lemma \ref{lem:tau_distortion}.
Thus we have $B _{n,\delexp} ( x )\subseteq B _{|I_n(x)|} (x) $ for all $x\in\J$ and $n\geqslant N$.
Next, let $v\in \J\setminus B _{\delexp,n}(x)$, so there is some $ 0\leqslant  j\leqslant n-1 $ such that $ d (\tau ^j x ,\tau ^j v ) >\delexp $.
We consider the following two cases, separately.
If $v\not\in I _n (x)$, then we have $ d ( x , v )\geqslant\dist (\partial I _n (x) , x )$.
If $v\in I _n (x)$, by the mean value theorem and Lemma \ref{lem:tau_distortion},
\[
	D |I_n(x)|^{-1} d ( x , v )\geqslant D |I_j(x)|^{-1} d ( x , v )\geqslant |(\tau ^j )'(\tilde{v} )| d ( x , v ) = d (\tau ^j x ,\tau ^j v ) >\delexp
\]
for some $\tilde{v}\in I _n (x )$.
Hence, together with the above observation, we have
\[
	B _{r _n (x)} (x)	\subseteq B _{n,\delexp} ( x )\subseteq B _{|I_n(x)|}(x)
\]
for all $x\in\J$ and $n\geqslant N$, where $r _n (x) :=\min\left\{\dist (\partial I _n (x) , x ) ,D ^{-1}\delexp |I_n(x)| \right\}$.
This finishes the proof together with the fact that
\[
	\lim _{n\to\infty}\frac{\log r _n (x)}{n} =\lim _{n\to\infty}\frac{\log |I_n (x)|}{n} =  -\int\log |\tau '|d\nu 
\]
for $\nu$-a.a. $x$, which follows from Birkhoff's ergodic theorem and Lemmas \ref{lem:tau_distortion}, \ref{lem:gibbs_props}.
\end{proof}

Let $\mathcal{C} _\varepsilon :=\mathcal{C} _{\varepsilon ,N_\varepsilon}$ for some fixed $N_\varepsilon$ so that it has a positive $\nu$ measure.
We consider then the restricted measure $\nu _\varepsilon :=\nu (\mathcal{C} _\varepsilon\cap\;\cdot\; ) $ and its lift
\[
	\mu _{\ttheta,\varepsilon} :=\nu\left\{ x\in\mathcal{C} _\varepsilon : ( x , W _\ttheta  ( x )  )\in\;\cdot\;\right\} .
\]
Recall that $s$-energy of a Borel measure $\mu $ on a metric space $ (\mathcal{E} , d _\mathcal{E} ) $ is defined by
\[
	I _s (\mu ) =\iint\frac{d\mu (x) d\mu (v)}{d _\mathcal{E} ( x,v) ^s} .
\]
Furthermore, $ I _s (\mu ) <\infty $ implies that $\underline{d} _\mu\geqslant s $ holds $\mu $-almost surely (see e.g. \cite[Section 4.3]{Falconer05}).

Observe that
\begin{eqnarray}\label{eq:E_calc}
	\int I _s (\mu _{\ttheta,\varepsilon })\, d\ttheta & = &	\iiint\frac{ d\nu _\varepsilon (x)\, d\nu _\varepsilon (v)\, d\ttheta }{(d(x,v) ^2 + ( W _\ttheta (x) - W _\ttheta ( v) ) ^2 ) ^{s/2} }\nonumber\\
	&\leqslant &\frac{1}{\delexp ^s}	+\iiint _{\{(x, v)\in\Delta _{\delexp}\} }\frac{ d\nu _\varepsilon (x)\, d\nu _\varepsilon (v)\, d\ttheta }{(d(x,v) ^2 + ( W _\ttheta (x) - W _\ttheta ( v) ) ^2 ) ^{s/2} }\nonumber\\
	& =: &	\frac{1}{\delexp ^s} + E _{\varepsilon, s}(\ttheta) ,	
\end{eqnarray}
where $\Delta _{\delexp} :=\{ (x, v)\in\J ^2 : d(x, v)\leqslant\delexp\} $ is the $\delexp $-diagonal set.

\begin{lemma}	\label{lem:potential_helper}
If $ E _{\varepsilon, s}(\ttheta) <\infty $ for some $s>0$ and $\varepsilon\in (0,1)$, then $\dim _H \mu _\ttheta\geqslant s$ for a.a. $\ttheta$.
\end{lemma}

\begin{proof}
In view of Lemma \ref{lem:Barreira_cite}, it suffices to show for a.a. $\ttheta$ that
\[
	\nu \{ x\in\mathcal{C} _\varepsilon :\underline{d} _{\mu _\ttheta} (x, W _\ttheta (x))<s\}=0.
\]
As $ E _{\varepsilon, s} <\infty  $, we have $ I _s (\mu _{\ttheta,\varepsilon }) <\infty $ by \eqref{eq:E_calc} for a.a. $\ttheta $. In the following, let $\ttheta $ be such a parameter.
Observe that the boundedness of $s$-energy implies that $\underline{d} _{\mu _{\ttheta,\varepsilon }} (x,y)\geqslant s $ for $\mu _{\ttheta,\varepsilon } $-almost all $(x,y)\in\J\times\mathbb{R} $, i.e. $\underline{d} _{\mu _{\ttheta,\varepsilon }} ( x , W _\ttheta (x) )\geqslant s $ for $\nu $-almost all $ x\in\mathcal{C} _\varepsilon $.
On the other hand, by Borel density theorem we have
\[
	\lim _{r\to 0}\frac{\mu _{\ttheta,\varepsilon } ( B _r (x,y))}{\mu _\ttheta ( B _r (x, y) )} =\lim _{r\to 0}\frac{\mu _\ttheta  ( B _r (x,y)\cap (\mathcal{C} _\varepsilon\times\R) )}{\mu _\ttheta ( B _r (x, y) )} = 1
\]
for $\mu _\ttheta $-almost all $ (x,y)\in\mathcal{C} _\varepsilon\times\R $, so $\underline{d} _{\mu _\ttheta} ( x , W _\ttheta (x) ) =\underline{d} _{\mu _{\ttheta,\varepsilon }} ( x , W _\ttheta (x) )\geqslant s $ for $\nu $-a.a. $ x\in\mathcal{C} _\varepsilon $.
\end{proof}

For any $(x,v)\in\J^2$, let $h _{x,v} $ denote the density function of the random variable $\ttheta\mapsto W _\ttheta (x)-W _\ttheta(v) $ with respect to the Lebesgue measure on $\R$.

The next key lemma is a slightly deformed \cite[Lemma 4.2]{Moss12}, which relies on a specific geometrical property deduced from the assumption of $g$ and the expansivity of $\tau$.

\begin{lemma}[Key lemma]	\label{lem:moss}
There is a constant $C _h >0$ such that
\[
	\|h _{x,v}\| _\infty\leqslant\frac{C _h}{\lambda ^n (x)}
\]
for all $v\not\in B _{n,\delexp}(x)\setminus B _{n+1,\delexp}(x)$, $n\in\N $ and $x\in\J$.
\end{lemma}

\begin{proof}
The proof of the aforementioned lemma can be directly translated with the following remarks:
\begin{itemize}
\item The points $ z , x, y $ in their proof correspond with $ x , x , v $  here.
\item Compared to $B _{n,\delexp}(x)\setminus B _{n+1,\delexp}(x)$, their counterpart $X ^r _n(z)$ is further restricted to $ B _{n,2\delexp}(z)\times B _{n,2\delexp}(z)$, however, this fact does not matter in this part.
\item We allow here possible non-differentiability of $\tau $ and $\lambda $ at the end points of Markov partition. This does not matter, as only $ g\in C ^\infty (\R /\Z ) $ is the crucial assumption here.
\end{itemize}
\end{proof}

\begin{lemma}	\label{lem:lower1}
Suppose that $ -\int\log\lambda\, d\nu < h _\tau (\nu ) $.
Then we have for a.a. $\ttheta$ that
\[
	\dim _H \mu _\ttheta\geqslant\dim _H\nu + 1 +\frac{\int\log\lambda\, d\nu}{\int\log |\tau '|\, d\nu}.
\]
\end{lemma}

\begin{proof}
We are going to show the assumption of Lemma \ref{lem:potential_helper}.
By Lemma \ref{lem:moss} and the substitution formula of integral, we have
\begin{eqnarray*}
	\int\frac{d\ttheta   }{(d(x,v) ^2 + ( W _\ttheta (x) - W _\ttheta ( v) ) ^2 ) ^{s/2} }	
	& = &	\int _\R\frac{h _{x, v} (z)\, d z   }{(d(x,v) ^2 + z ^2 ) ^{s/2} }\\
	& = &	d(x,v) ^{1-s}\int _\R\frac{h _{x,v} ( d(x,v) t )\, d t}{( 1 + t ^2) ^{s/2}}\\
	&\leqslant &	d(x,v) ^{1-s}\int _\R\frac{d t}{( 1 + t ^2) ^{s/2}}\,\| h _{x, v}\| _\infty\\
	& = &	d(x,v) ^{1-s}\, K _s\,\| h _{x, v}\| _\infty\quad\leqslant\quad 	\frac{K _s\,  C _h}{\lambda ^n (x)\, d(x,v) ^{s -1} } ,
\end{eqnarray*}
where $K _s :=\int _\R\frac{d t}{( 1 + t ^2 ) ^{s/2}}$ is well-defined for $s>1$.
Let $\varepsilon > 0 $ be small enough and let $ x\in\mathcal{C} _\varepsilon $.
By Fubini's lemma we have
\begin{eqnarray*}
		&& \iint _{B _\delexp (x)}\frac{d\nu _\varepsilon (v) }{(d(x,v) ^2 + ( W _\ttheta (x) - W _\ttheta ( v) ) ^2 ) ^{s/2} }\, d\ttheta \\
		&=&	\sum _{n =0} ^\infty\iint _{B _{n,\delexp} (x)\setminus B _{n+1,\delexp} (x)}\frac{d\nu _\varepsilon (v) }{(d(x,v) ^2 + ( W _\ttheta (x) - W _\ttheta ( v) ) ^2 ) ^{s/2} }\, d\ttheta\\
	&\leqslant &	K _s\, C _h\,\sum _{n=0} ^\infty\frac{1}{\lambda ^n (x)}\int _{B _{n,\delexp} (x)\setminus B _{n+1,\delexp} (x)}\frac{d\nu _\varepsilon (v)}{d(x,v) ^{s-1}}\\
		&\leqslant &	K _s\, C _h\,\sum _{n=0} ^\infty\frac{1}{\lambda ^n (x)}\int _{ B _{e ^{n(-\int\log |\tau '| d\nu +\varepsilon)} }(x)\setminus  B _{e ^{n (-\int\log |\tau '| d\nu -\varepsilon)} } (x) }\frac{d\nu _\varepsilon (v)}{d(x,v) ^{s-1}}\\
		&\leqslant &	K _s\, C _h\,\sum _{n=0} ^\infty\frac{1}{\lambda ^n (x)}\frac{\nu ( B _{e ^{n (-\int\log |\tau '| d\nu -\varepsilon)} } (x) ) }{e ^{n(s-1)(-\int\log |\tau '| d\nu +\varepsilon)} }\\
		&\leqslant &	K _s\, C _h\,\sum _{n=0} ^\infty\frac{\left( e ^{n(-\int\log |\tau '| d\nu +\varepsilon)}\right) ^{\dim _H\nu -\varepsilon} }{e ^{ n (\int\log\lambda d\nu -\varepsilon) + n (s-1) (-\int\log |\tau '| d\nu -\varepsilon)} } <\infty
\end{eqnarray*}
for all $s\in ( 1 , S (\varepsilon ) )$, where
\[
	  S (\varepsilon) := 1 +\frac{(\int\log |\tau '| d\nu -\varepsilon) (\dim _H\nu -\varepsilon) +\int\log\lambda d\nu -\varepsilon }{\int\log |\tau ' | d\nu +\varepsilon } .
\]
Observe that $\lim _{\varepsilon\to 0} S (\varepsilon )  =\dim _H\nu + 1 +\frac{\int\log\lambda d\nu}{\int\log |\tau '| d\nu} > 1$ by the assumption, so the $(1,S(\varepsilon))$ is not empty.
Finally, integrating the above inequalities related to $x$ by $\nu _{\ttheta ,\varepsilon}$, we obtain $ E _{\varepsilon , s} <\infty $ for all $ s\in ( 1 , S (\varepsilon ) ) $, and thus $\dim _H \mu _\ttheta\geqslant S(\varepsilon)$ by Lemma \ref{lem:potential_helper}.
Letting $\varepsilon\to 0$ finishes the proof.
\end{proof}

\begin{lemma}	\label{lem:lower2}
Suppose that $ -\int\log\lambda\, d\nu\geqslant h _\tau (\nu ) $.
Then we have $\mu _\ttheta$-a. that
\[
	\underline{d} _{\mu _\ttheta }\geqslant\frac{h _\tau (\nu )}{-\int\log \lambda \,d\nu}.
\]
\end{lemma}

\begin{proof}

For $ s < s' < 1 $, by Jensen's inequality and Lemma \ref{lem:moss}, we have
\begin{eqnarray*}
	&&	\int\frac{d\ttheta   }{(d(x,v) ^2 + ( W _\ttheta (x) - W _\ttheta ( v) ) ^2 ) ^{s/2} }		\\
	&\leqslant &	\left(\int\frac{d\ttheta   }{(d(x,v) ^2 + ( W _\ttheta (x) - W _\ttheta ( v) ) ^2 ) ^{s/(2s')} }\right)^{s'}		\\
		& = &	\left(\int _\R\frac{h _{x, v} (z)\, d z   }{(d(x,v) ^2 + z ^2 ) ^{s/(2s')} }\right)^{s'}\\
	&\leqslant &	\left(\int _\R\frac{d z   }{(d(x,v) ^2 + z ^2 ) ^{s/(2 s')} }\right) ^{s'}\| h _{x, v}\| _\infty ^{s'}\\
	& = &	d ( x , v ) ^{s' - s}\left(\int _\R\frac{d t   }{(1 + t ^2 ) ^{s/(2 s')} }\right) ^{s'}\| h _{x, v}\| _\infty ^{s'}\\
	& = &	K _{s' / s}\, d ( x , v ) ^{s' - s}\| h _{x, v}\| _\infty ^{s'}\quad\leqslant 	\frac{C _h  ^{s'}\, K _{s' / s} }{(\lambda ^n (x) ) ^{s'}} , 
\end{eqnarray*}
where $ K _t $ is the constant in proof of Lemma \ref{lem:lower1} and well-defined for $ t = s'/s > 1 $.
Note that we also used $d(x,v)\leqslant 1$.
Let $\varepsilon > 0 $ be small enough and let $ x\in\mathcal{C} _\varepsilon $.
By Fubini's lemma we have
\begin{eqnarray*}
	&& \iint _{B _\delexp (x)}\frac{d\nu _\varepsilon (v) }{(d(x,v) ^2 + ( W _\ttheta (x) - W _\ttheta ( v) ) ^2 ) ^{s/2} }\, d\ttheta \\
	&=&	\sum _{n =0} ^\infty\iint _{B _{n,\delexp} (x)\setminus B _{n+1,\delexp} (x)}\frac{d\nu _\varepsilon (v) }{(d(x,v) ^2 + ( W _\ttheta (x) - W _\ttheta ( v) ) ^2 ) ^{s/2} }\, d\ttheta\\
	&\leqslant &	C _h  ^{s'}\, K _{s' / s}\,\sum _{n=0} ^\infty\frac{\nu ( B _{n,\delexp} (x)  )}{(\lambda ^n (x) ) ^{s'}}\\
		&\leqslant &	C _h  ^{s'}\, K _{s' / s}\,\sum _{n=0} ^\infty\frac{\nu ( B _{e ^{n(-\int\log |\tau '| d\nu +\varepsilon)}} (x)  )}{(\lambda ^n (x) ) ^{s'}}\\ 
	&\leqslant &	C _h  ^{s'}\, K _{s' / s}\,\sum _{n=0} ^\infty e ^{n(-\int\log |\tau '| d\nu +\varepsilon) (\dim _H\nu -\varepsilon)} e ^{- s' n (\int\log\lambda\, d\nu -\varepsilon )} <\infty
\end{eqnarray*}
for all $ s\in ( 0 , 1 ) $ and $ s'\in\left(\max\{ s,\tilde{S} (\varepsilon )\} , 1\right) $, where
\[
	\tilde{S} (\varepsilon ) :=\frac{(\int\log |\tau '|\, d\nu -\varepsilon) (\dim _H\nu -\varepsilon)}{-\int\log\lambda\, d\nu +\varepsilon }  .
\]
Observe that $\tilde{S} (\varepsilon ) < 1 $ for $\varepsilon > 0 $ so that the following approximation works, even in case $\tilde{S} ( 0 ) = 1 $.
Since $\lim _{\varepsilon\to 0}\tilde{S} (\varepsilon ) =\frac{h _\tau (\nu )}{\int\log |\tau '| d\nu}<1$, the claim follows by lemmas \ref{lem:C_full} and \ref{lem:potential_helper}.
\end{proof}

\begin{proof}[Conclusion of proof of Lemma \ref{lem:lift}]
For almost all $\ttheta$, it is a combination of Lemmas \ref{lem:lower1} and \ref{lem:lower2} that
\[
	\underline{d} _{\mu _\ttheta}\geqslant\min\left\{\dim _H\nu + 1 +\frac{\int\log\lambda\, d\nu}{\int\log |\tau '|\, d\nu} ,\frac{h _\tau (\nu )}{\int\log |\tau '|\, d\nu}\right\} 
\]
holds $\mu _\ttheta $-almost surely.
Hence, together with the upper bound obtained in Lemma \ref{lem:upper_bound}, the proof of Lemma \ref{lem:lift} is finished.
\end{proof}

\section{Proof of Lemma \ref{lem:Bedford_weak}} \label{sec:proof_Bed_weak}

We mainly follow the major arguments in \cite{Bedford89} except for the part of Lemma \ref{lem:Bed_estlem_3}, the novel lower bound for the oscillation of randomised graph.
All things become a bit more tidy than in the mentioned work because of the additional variable $\ttheta $.
Indeed, we consider the following dynamical system.
For each $ (\vartheta , i )\in\T\times\Sigma _\ell$, the contraction $ F _{\vartheta, i } :\T\times\R\to\T\times\R$ is defined by
\[
	F _{\vartheta, i } ( x, y) =\left(\rho _i (x) ,\lambda (\rho _i (x) )\, y + g (\rho _i (x) +\vartheta  )\right) .
\]
Observe that
\begin{equation}
	F _{\vartheta _0 ,\kappa (x) } (\tau x , W _{\sigma\ttheta} (\tau x ) ) =\left( x , W _\ttheta ( x )\right) 	\label{eq:F_invariant}
\end{equation}
for all $ x\in\T$ and $\ttheta\in \T ^{\N _0}$.
In addition, the derivative matrix for $F _{\vartheta ,i}$ can be calculated as
\begin{eqnarray*}
	D F _{\vartheta , i} ( x , y )	& = &	\left[
	\begin{matrix}
		\rho _i ' (x)	&	0\\
		( y\cdot\lambda ' + g' (\cdot +\vartheta) ) (\rho _i (x) )\cdot\rho _i ' ( x )	&	\lambda (\rho _i (x) )
	\end{matrix}	\right]\\
	& =: &	\left[
	\begin{matrix}
		a _i (x)	&	0\\
		b _i (\vartheta , x , y )	&	c _i ( x )
	\end{matrix}	\right] 
\end{eqnarray*}
for $ (x,y)\in\T\times\R$.
It is convenient for iterated maps to be denoted by
\[
	F _{\ttheta ,\mathbf{i}} :=  F _{\vartheta _{n-1} , i _{n-1}}\circ\cdots\circ  F _{\vartheta _1 , i _1}\circ F _{\vartheta _0 , i _0}
\]
for $ (\ttheta ,\mathbf{i} )\in\T ^n\times\Sigma _\ell ^n $.
Observe that, applying \eqref{eq:F_invariant} iteratively, we have
\begin{equation}
	F _{[\ttheta] _n , [x] _n}\left(\tau ^n x , W _{\sigma ^n\ttheta} (\tau ^n x)\right) =\left( x , W _\ttheta (x)\right)  \label{eq:F_invariant_n}
\end{equation}
for all $ x\in\J $, $\ttheta\in \T ^{\N _0} $ and $ n\in\N $, where $ [\ttheta ] _n := (\vartheta _0 ,\vartheta _1 ,\ldots ,\vartheta _{n-1} ) $.

In the following, a continuous map $ C : I\to [0,1]\times\R $ for any subinterval $I\subseteq\T$ will be called a curve if $C_1$ is strictly monotone, where $C (t) = ( C _1(t) , C _2(t) )$.
Given a curve $ C : I\to\T\times\R $, its width, height and the hight over $\J$ are, respectively, denoted by
\[
	| C | _W := | C_1(I) |, \quad | C | _H :=\sup _{t_1, t_2\in I} | C _2 (t_1) - C _2 (t_2) |  \quad\mbox{and}\quad  | C | _{\J, H} :=\sup _{t_1, t_2\in I\cap C_1^{-1}(\J )} | C _2 (t_1) - C _2 (t_2) | .
\]

\begin{lemma}	\label{lem:Bed_estlem_1}
Let $ M > 0 $ and $ I\subseteq\T $ be a subinterval.
For any curve $ C : I\to\T\times [-M, M]$ with $C_1$, we have
\[
	\inf _{t\in I} c _i (C _1 (t))\cdot | C | _{\J, H} - b\cdot | C | _W\leqslant | F _{\vartheta , i}\circ C | _{\J, H}
\quad\mbox{and}\quad
	| F _{\vartheta , i}\circ C | _H\leqslant\sup _{t\in I} c _i (C _1 (t))\cdot | C | _{\J, H} + b\cdot | C | _W,
\]
where $ b :=\max\left\{ b _i (\vartheta , x , y ) :\begin{array}{l}
	 i\in\Sigma _\ell ,\;\vartheta\in\T ,\\  (x,y)\in\T\times [-M, M]
\end{array}	\right\} $.
\end{lemma}

\begin{proof}
Due to the monotonicity assumption on $C_1$, we can reparametrise the curve so that $C_1(t)=t$ by implicit function theorem.
Observe that after this step, it hold $|C|_W=|I|$ and $|C|_{\J ,H}=\sup _{t_1,t_2\in I\cap\J}|C_2 (t_1) - C_2(t_2)|$.
Furthermore, by Stone-Weierstrass approximation theorem, we may assume $C_2\in C^1(\R )$.
Now, we can write
\[
	(F _{\vartheta , i}\circ C) ' (t) = D F _{\vartheta , i} ( C (t) )\cdot\left[\begin{matrix}
	C _1 ' (t)\\ C _2 ' (t)
	\end{matrix}\right]
	 =\left[\begin{matrix}
	 t\\ b _i (\vartheta , C (t) ) + c _i ( t)\, C _2 ' (t)
	\end{matrix}\right] 
\]
for all $ t\in I $, so we have
\begin{eqnarray*}
	| F _{\vartheta , i}\circ C | _{\J, H}	&\geqslant &\int _{t_1} ^{t_2}  b _i (\vartheta , C (t) ) + c _i (t)\, C _2 ' (t)\, d t\\
	&\geqslant &	\inf _{t\in [t_1, t_2]} c _i (t)\int _{t _1} ^{t _2}  C _2 ' (t)\, d t - b\cdot (t_2 - t_1) 	\\
	&\geqslant &	\inf _{t\in I} c _i (t)\cdot\left( C _2  (t _1) - C _2 ( t _2 )\right)  - b\, | C | _W
\end{eqnarray*}
for all $ t _1, t _2\in\J\cap I $.
The first inequality in the claim follows from this, immediately.
Moreover, there exist some $t _1, t _2\in I$ such that
\begin{eqnarray*}
	| F _{\vartheta , i}\circ C | _H &=& F _{\vartheta , i}\circ C(t_2) - F _{\vartheta , i}\circ C (t_1) \\
	& = &\int _{t_1} ^{t_2}  b _i (\vartheta , C (t)) + c _i (t)\, C _2 ' (t)\, d t \quad\leqslant\quad	b\, | C | _W +\int _{t_1} ^{t_2} c _i (t)\, C _2 ' (t)\, d t .
\end{eqnarray*}
Finally, since $c_i>0$, by the mean value theorem, there is some $ t _3\in ( t_1, t_2 ) $ such that
\[
	\int _{t_1} ^{t_2} c _i (t)\, C _2 ' (t)\, d t = c _i (t _3)\int _{t_1} ^{t_2} C _2 ' (t)\, d t\leqslant	\sup _{t\in I} c_i (t)\, | C _2 | _{\J, H}.
\]
\end{proof}

\begin{lemma}\label{lem:Bed_estlem_2}
Given $ M > 0 $, there are constants $ L ' , L '', L '''> 0 $ such that
\[
	\left( L '\, | C | _{\J, H} - L ''\, | C | _W\right)\,\lambda ^n (\rho _{\mathbf{i}} (x) )\leqslant | F _{\ttheta ,\mathbf{i}} ^n\circ C | _{\J, H}\leqslant L '''\lambda ^n (\rho _{\mathbf{i}} (x) )  
\]
for any curve $ C : I\to\T\times [-M, M] $, $  (\ttheta ,\mathbf{i} )\in\T ^n\times\Sigma _\ell ^n  $, $ x\in\T $ and $ n\in\N $.
\end{lemma}

\begin{proof}
Let $C, (\ttheta ,\mathbf{i} ), x$ and $n$ be given as stated.
By applying Lemma \ref{lem:Bed_estlem_1}, inductively, we obtain
\[
	| F _{\ttheta ,\mathbf{i}}\circ C | _{\J, H}\geqslant\hat{c} _0\cdots\hat{c} _{n-1}\, | C | _{\J, H} - b\,\sum _{j=1} ^n |\rho _{ [\mathbf{i} ]_{n-j}}\circ C_1 ( I ) | \prod _{k=n-j+1} ^{n-1}\hat{c} _k , 
\]
where $\hat{c} _k :=\inf _{t\in I} c _{i _k} (\rho _{[\mathbf{i}] _k}\circ C _1 (t) ) $ and $\prod _{k=n} ^{n-1}\hat{c} _k := 1 $.
Let $ ( x _k ) _{k=0} ^{n-1}\subset\T $ be such that
\[
	\hat{c} _k =\inf _{t\in I}\lambda\left(\rho _{[\mathbf{i}] _{k+1}}\circ C _1 (t)\right) =\lambda\left(\rho _{[\mathbf{i}] _{k+1}} ( x _k )\right) .
\]
For $p =0,\ldots, n-1$, since
\begin{eqnarray*}
	\left|\log\prod _{k=p} ^{n-1}\frac{\lambda\left(\rho _{[\mathbf{i}] _{k+1}} ( x _k )\right) }{\lambda\left(\rho _{[\mathbf{i}] _{k+1}} ( x  )\right)  }\right| 	&\leqslant &	\sum _{k=p} ^{n-1}\|  (\log\lambda ) '\| _\infty |\rho _{[\mathbf{i}] _{k+1}} ( x _k )  -\rho _{[\mathbf{i}] _{k+1}} ( x ) |\\
	&\leqslant &\frac{\|  (\log\lambda ) '\| _\infty  }{1 -\max _i\| (1 /\tau' )\circ \rho _i\| _\infty}	\quad =:\log K _0 ,
\end{eqnarray*}
we have
\begin{equation*}
	\prod _{k=p} ^{n-1}\hat{c} _k =\lambda ^{n-p} (\rho _{\mathbf{i}} (x) )\prod _{k=p} ^{n-1}\frac{\lambda\left(\rho _{[\mathbf{i}] _{k+1}} ( x _k )\right) }{\lambda\left(\rho _{[\mathbf{i}] _{k+1}} ( x  )\right)  }\quad\in [ K _0 ^{-1} , K _0 ]\cdot\lambda ^{n-p} (\rho _{\mathbf{i}} (x) ) .
\end{equation*}
Similarly, $ |\rho _{[\mathbf{i}] _p}\circ C_1( I ) |=|\rho _{[\mathbf{i}] _p} ' (u _p )|\, |C| _W=|1/\tau '|^p\left(\rho _{[\mathbf{i}] _p} (u _{n-j} )\right)\cdot |C| _W$ for some $ u _p\in [0,1] $, and there is a constant $ K _1 > 0 $ such that
\[
	|1/\tau '|^p\left(\rho _{[\mathbf{i}] _p} (u _{n-j} )\right) = |1/\tau '|^p\left(\rho _{[\mathbf{i}] _p} ( x )\right)\prod _{k=0} ^{p-1}\frac{|1/\tau'|\left(\rho _{[\mathbf{i}] _k} ( u _k )\right) }{|1/\tau'|\left(\rho _{[\mathbf{i}] _k} ( x )\right) }\in [ K _1 ^{-1} , K _1 ]\cdot  |1/\tau'| ^p\left(\rho _{[\mathbf{i}] _p} ( x )\right)
\]
for $p=0 ,\ldots , n-1$.
Observe that since the derivative of $\log |\tau'| $ on $\bigcup _{i\in\Sigma _\ell}I_i$ may fail to exist, we need here to consider the $\alpha $-Hölder semi-norm $|\cdot | _\alpha$ in order to obtain the constant $\log K _1 := |\log |\tau' || _\alpha / ( 1 -\| 1 /\tau'\| _\infty ^\alpha )$.
Consequently, the lower bound follows as
\begin{eqnarray*}
	| F _{\ttheta ,\mathbf{i}}\circ C | _{\J, H}	 &\geqslant &	K _0^{-1}\lambda ^{n-1} (\rho _{\mathbf{i}} (x) )\, | C | _{\J, H} - b\, K _0\, K _1\, |C|_W\sum _{j=1} ^n    |1/\tau'| ^{n-j}\left(\rho _{[\mathbf{i}] _{n-j}} ( x )\right)\lambda ^{j-1} (\rho _{\mathbf{i}} (x) )\\
	& = &	\lambda ^n (\rho _{\mathbf{i}} (x) )\,\left(\frac{K _0^{-1}| C | _{\J, H}}{\lambda (\rho _{i _{n-1}} (x) )}  -  b\, K _0\, K _1\,\lambda (\rho _{[\mathbf{i}] _{n-j}} (x) )\, |C|_W\sum _{j=0} ^{n-1}    (1/|\tau'\lambda|) ^{n-j}\left(\rho _{[\mathbf{i}] _{n-j}} ( x )\right)\right)\\
	&\geqslant &	\lambda ^n (\rho _{\mathbf{i}} (x) )\,\left(K _0^{-1}\inf ( 1 /\lambda )\,  | C | _{\J, H} -\frac{b\, K _0\, K _1\,\|\lambda '\| _\infty }{1 -\max _i\| (1 /\tau '\lambda )\circ\rho _i\| _\infty }\, | C | _W\right) .
\end{eqnarray*}
Finally, the upper bound can be derived in a similar manner.
Indeed, as done above for $\hat{c}_k$, we can show
\[
	\prod _{k=p} ^{n-1}\bar{c} _k\in [ K _0 ^{-1} , K _0 ]\cdot\lambda ^{j-1} (\rho _{\mathbf{i}} (x) )  ,
\]
for $p=0,\ldots ,n-1$, where $\bar{c} _k :=\sup _{t\in I} c _{i _k} (\rho _{[\mathbf{i}] _k}\circ C _1 (t) ) $.
From this follows that
\begin{eqnarray*}
	| F _{\ttheta ,\mathbf{i}}\circ C | _{\J, H}	 &\leqslant &	\lambda ^{n-1} (\rho _{\mathbf{i}} (x) )\, | C | _{\J, H} + b\, K _0\, K _1\, |C|_W\sum _{j=1} ^n    |1/\tau'| ^{n-j}\left(\rho _{[\mathbf{i}] _{n-j}} ( x )\right)\lambda ^{j-1} (\rho _{\mathbf{i}} (x) )\\
	&\leqslant &	\lambda ^n (\rho _{\mathbf{i}} (x) )\,\left(\| 1 /\lambda\| _\infty  +\frac{b\, K _0\, K _1\,\|\lambda '\| _\infty }{1 -\max _i\| (1 /\tau '\lambda)\circ\rho _i\| _\infty }\right) .
\end{eqnarray*}
\end{proof}

\begin{lemma}	\label{lem:Bed_estlem_3}
Assume that $ W _\ttheta $ is non-degenerate for almost all $\ttheta $.
Then, for any $L>0$, we have
\[
	\lim _{n\to\infty}\frac{\log a _L (\sigma ^n\ttheta )}{n} =  0
\]
for a.a. $\ttheta$, where
$ a _L (\ttheta ) := \sup _{u, v\in\J} | W _\ttheta ( u ) - W _\ttheta (v) | - L\, | u - v |$.
\end{lemma}

\begin{proof}
Let $L>0$ be arbitrary.
Observe that $a _L(\ttheta )>0$ for those parameters $\ttheta$ for which $ W _\ttheta $ is non-degenerate.
Let $ M :=\sup _{\ttheta , x\in\T ^{\N _0}\times\J} | W _\ttheta (x)|$, and recall $c_i:=\lambda\circ\rho _i$.
Given $u, v\in\T$ and $ i\in\Sigma _\ell$, the invariance relation of \eqref{eq:F_invariant} yields
\[
	W _\ttheta (\rho _i (u) ) = c _i (u)\, W _{\sigma\ttheta} ( u ) + g(\rho _i (u) +\vartheta _0) .
\]
On the other hand, there is some $w$ between $u,v$ such that $|\rho _i (u) -\rho _i (v)| = |\rho _i ' (w) |\, |u - v|$, so we have
\begin{eqnarray*}
	& &	\left|c _i(u)\,	 W _{\sigma\ttheta} ( u ) -c _i ( v )\, W _{\sigma\ttheta} (v)\right|\\
	&=&	\left|c _i (w)\left( W _{\sigma\ttheta}(u)- W _{\sigma\ttheta} (v)\right) +W _{\sigma\ttheta}(u)\left(  c _i (u) - c _i (w) \right)+W _{\sigma\ttheta}(v)\left(c _i (w) -c_i (v)  \right)\right|\\
	&\geqslant &	c _i (w)\,  | W _{\sigma\ttheta} ( u ) -  W _{\sigma\ttheta} (v) |  - 2 M\,\|\lambda '\| _\infty\,  | u - v |.
\end{eqnarray*}
Together with the facts
\[
	- | g(\rho _i (u) +\vartheta _0) - g(\rho _i (v) +\vartheta _0) |\geqslant -\| g'\| _\infty\, | u - v | ,
\]
and
\[
	c _i (w)-|\rho _i ' ( w )|  =c _i (w) \left( 1 -\frac{1}{|\lambda\tau ' | (\rho _i (w) )}\right)\geqslant (\inf\lambda )\left( 1 -\max _j \|(1/\tau' \lambda)\circ\rho _j \| _\infty\right) =:\delta _1 > 0 ,
\]
it follows that
\begin{eqnarray*}
	 &  &	| W _\ttheta (\rho _i ( u ) ) - W _\ttheta (\rho _i ( v )) | - L\, |\rho _i ( u ) -\rho _i ( v ) |\\
	& = &	|c _i(u)W _{\sigma\ttheta}(u) - c _i(v)W _{\sigma\ttheta}(v)+g(\rho _i (u) +\vartheta _0) -g(\rho _i (v) +\vartheta _0)| - L\, |\rho _i'(w)|\,|u-v|\\
	 &\geqslant &	c _i (w)\,\left( | W _{\sigma\ttheta } ( u ) -  W _{\sigma\ttheta } ( v )  | - L\, | u - v |\right)\\
	 &&	\qquad  +  ( L\,c _i (w) - L\,|\rho _i ' ( w )|  - 2 M\,\|\lambda '\| _\infty -\| g'\| _\infty)\, | u - v |\\
	 &\geqslant &		c _i (w) \left( | W _{\sigma\ttheta } ( u ) -  W _{\sigma\ttheta } ( v )  | - L\, | u - v |\right) + ( L\,\delta _1 -  2 M\,\|\lambda '\| _\infty -\| g'\| _\infty )\, | u - v |\\
	 &\geqslant &	c _i (w) \left( | W _{\sigma\ttheta } ( u ) -  W _{\sigma\ttheta } ( v )  | - L\, | u - v |\right) ,
\end{eqnarray*}
for all $ L\geqslant L_0:= ( 2 M\,\|\lambda '\| _\infty +\| g'\| _\infty )\delta _1 ^{-1}$.
In particular, for those $\ttheta$ such that $a_L(\sigma\ttheta)>0$,
\[
	a _L(\ttheta )\geqslant\sup _{u, v\in\J\cap I _i}| W _\ttheta ( u ) - W _\ttheta ( v ) | - L\, | u - v | \geqslant (\inf c_i) \, a_L(\sigma\ttheta )
\]
Now, we prove the lemma for $L\geqslant L _0$.
Fix any $i$, say $i=0$.
We have shown that $ \log a _L (\ttheta ) - \log a _L(\sigma\ttheta )\geqslant\log(\inf c_0)$ for a.a. $\ttheta$, which implies $\log a_L - \log a_L\circ\sigma\in L ^1_{m ^{\N _0}}$ with $\int  a _L (\ttheta ) - a _L (\sigma\ttheta  )\, d m ^{\N _0} (\ttheta) = 0 $ in view of \cite[Lemma 2]{Keller96}.
Thus, by the telescoping sum and Birkhoff's ergodic theorem, we have
\[
	\lim _{n\to\infty}\frac{\log a _L (\sigma ^n\ttheta)}{n} = \lim _{n\to\infty}\frac{1}{n}\sum _{k=1} ^n \log a _L (\sigma ^k\ttheta ) - \log a _L (\sigma ^{k-1}\ttheta ) = 0
\]
for a.a. $\ttheta $.
Finally, the claim holds also for $L\in(0,L_0)$ as $a_{L_0}(\ttheta)\leqslant a_L(\ttheta)\leqslant 2 M + 1$ for all $\ttheta$.
\end{proof}

\begin{proof}[Proof of Lemma \ref{lem:Bedford_weak}]
Let $ L' , L'' , L''' > 0 $ be the constants of Lemma \ref{lem:Bed_estlem_2} related to the given $M :=\sup _{(\ttheta ,x)\in\T^{\N _0}\times\T} | W _\ttheta (x) | $.
For the upper bound, let $ x\in\J $, $ n\in\N $ and $\ttheta\in\T ^{\N _0} $ be arbitrary.
Consider the curve $ C :\T\to\T\times\R $ by $ C ( t ) := (t , W _{\sigma ^n\ttheta} (t))$.
The invariance relation of \eqref{eq:F_invariant_n} yields
\begin{equation}
	( t , W _\ttheta ( t ) ) = F  _{[\ttheta ] _n , [x] _n}\left( C (\tau ^n t )\right)	\label{eq:F_invariant_in_proof}
\end{equation}
for all $t\in I _n (x)$.
Thus, by Lemma \ref{lem:Bed_estlem_2} we have
\begin{eqnarray*}
	\sup _{v\in I _n(x)} | W _\ttheta ( x ) - W _\ttheta ( v ) |  &\leqslant &	\sup _{u, v\in I _n(x)} | W _\ttheta ( u ) - W _\ttheta ( v ) |\\
	& = &\left|  F _{[\ttheta ] _n , [x] _n}\circ C\right| _H \leqslant\quad   L '''\,\lambda ^n (\rho _{[x] _n} (\tau ^n x))\quad =  L '''\,\lambda ^n (x) .
\end{eqnarray*} 
Therefore, the first claim is satisfied with $\overline{C} :=  L ''' $.

Next, we show the last claim.
Suppose that $ W _\ttheta $ is non-degenerate for almost all $\ttheta $.
Let $a _L :\T ^{\N _0}\to[0,\infty )$ be defined as in Lemma \ref{lem:Bed_estlem_3} for $L := L'' / L '$.
We claim that $c (\ttheta ) := ( L ' / 2 )\, a (\ttheta)$ satisfies the claimed properties.
By construction, it only remains to verify the claimed lower bound inequality.
Let $ x\in\J $ and $ n\in\N $.
Further, let $ C $ be the curve as above.
As $W _{\sigma ^n\ttheta}$ is continuous on a compact set $\J$, there are $ t _1 , t _2\in\J $ such that
\[
	a _L (\sigma ^n\ttheta )  = | W _{\sigma ^n\ttheta} ( t_1 ) - W _{\sigma ^n\ttheta} (t_2) | - L\, | t_1 - t_2 |  .
\]
Let $\hat{C} $ be the restriction of $ C $ on the interval $ [ t _1 , t _2 ] $, which forms again another curve.
By Lemma \ref{lem:Bed_estlem_2} and the relation \eqref{eq:F_invariant_in_proof}, we have
\begin{eqnarray*}
	c (\sigma ^n\ttheta )\,\lambda ^n (x)	& = &	( L' / 2 )\,\left( | W _{\sigma ^n\ttheta} ( t_1 ) - W _{\sigma ^n\ttheta} (t_2) | - L\cdot | t_1 - t_2 |\right)\,\lambda ^n ( x )\\
	&= &		2 ^{-1}\,\left( L' | W _{\sigma ^n\ttheta} ( t_1 ) - W _{\sigma ^n\ttheta} (t_2) | - L ''\cdot | t_1 - t_2 |\right)\,\lambda ^n (x )\\
	&\leqslant &	2 ^{-1}\left( L'\, |\hat{C} | _{\J, H} - L''\, |\hat{C} | _W\right)\,\lambda ^n\left(\rho _{[x] _n} (\tau ^n x )\right)\\
	&\leqslant &	2 ^{-1}\left| F _{[\ttheta] _n, [x] _n} ^n\circ\hat{C}\right| _{\J, H}\\
	&\leqslant &	2 ^{-1}\left| F _{[\ttheta] _n, [x] _n} ^n\circ C\right| _{\J, H}\\
	& = &	2 ^{-1}\sup _{u, v\in\J\cap I _n(x)} | W _\ttheta ( u ) - W _\ttheta ( v ) |\quad\leqslant \sup _{v\in\J\cap I _n(x)} | W _\ttheta ( x ) - W _\ttheta ( v ) | .
\end{eqnarray*}
Finally, the second claim follows also the above estimate.
Indeed, if $W=W_{\mathbf{0}}$ is non-degenerate, then we can choose $c:=c(\sigma ^n \mathbf{0})=(L'/2)a_L(0)>0$.
\end{proof}

\bibliography{dynamical_copy}

\begin{thebibliography}{MW12}

\bibitem[Bar08]{Barreira08}
Luis Barreira.
\newblock {\em Dimension and Recurrence in Hyperbolic Dynamics}.
\newblock Springer, 2008.

\bibitem[Bed89]{Bedford89}
Tim Bedford.
\newblock The box dimension of self-affine graphs and repellers.
\newblock {\em Nonlinearity}, 2:53--71, 1989.

\bibitem[Bá15]{Barani14}
Balázs Bárány.
\newblock On the {L}edrappier-{Y}oung formula for self-affine measures.
\newblock {\em Mathematical Proceedings of the Cambridge Philosophical
  Society}, 2015.

\bibitem[Fal05]{Falconer05}
Kenneth Falconer.
\newblock {\em Fractal Geometry: Mathematical Foundations and Applications}.
\newblock John Wiley \& Sons, Ltd, 2 edition, 2005.

\bibitem[Har16]{Hardy16}
G.~H. Hardy.
\newblock {W}eierstrass's non-differentiable function.
\newblock {\em Transactions of the American Mathematical Society}, 17(3):pp.
  301--325, 1916.

\bibitem[Hun98]{Hunt98}
Brian~R. Hunt.
\newblock The {H}ausdorff dimension of graphs of {W}eierstrass functions.
\newblock {\em Proceedings of the {A}merican mathematical society},
  126:791--800, 1998.

\bibitem[Jin11]{Jin11}
Xiong Jin.
\newblock The graph and range singularity spectra of b-adic independent cascade
  functions.
\newblock {\em Advances in Mathematics}, 226(6):4987 -- 5017, 2011.

\bibitem[JS15]{JS15}
Johannes Jaerisch and Hiroki Sumi.
\newblock Multifractal formalism for expanding rational semigroups and random
  complex dynamical systems.
\newblock {\em Nonlinearity}, 28(8):2913, 2015.

\bibitem[Kel96]{Keller96}
Gerhard Keller.
\newblock A note on strange nonchaotic attractors.
\newblock {\em Fund. Math}, 151:139--148, 1996.

\bibitem[Kel98]{Keller98}
Gerhard Keller.
\newblock {\em Equilibrium States in Ergodic Theory}.
\newblock London Mathematical Society Student Texts. Cambridge University
  Press, 1998.

\bibitem[Kel17]{Keller2017}
Gerhard Keller.
\newblock A simpler proof for the dimension of the graph of the classical
  weierstrass function.
\newblock {\em Ann. Inst. H. Poincaré Probab. Statist.}, 53(1):169--181, 02
  2017.

\bibitem[MW12]{Moss12}
A~Moss and C~P Walkden.
\newblock The {H}ausdorff dimension of some random invariant graphs.
\newblock {\em Nonlinearity}, 25:743–760, 2012.

\bibitem[Pes97]{Pesin97}
Yakov~B. Pesin.
\newblock {\em Dimension Theory in Dynamical Systems}.
\newblock University of Chicago Press, 1997.

\bibitem[She15]{Shen15}
Weixiao Shen.
\newblock {H}ausdorff dimension of the graphs of the classical {W}eierstrass
  functions.
\newblock {\em arXiv}, 2015.

\bibitem[Tod15]{Todorov15}
Dmitry Todorov.
\newblock Hölder properties of {W}eierstrass-like solutions of
  $\theta$-twisted cohomological equations.
\newblock {\em arxiv}, 2015.

\end{thebibliography}
\bibliographystyle{alpha}

\end{document}